\documentclass[final,hidelinks,onefignum,onetabnum]{siamart250211}

\usepackage{lipsum}
\usepackage{amsfonts}
\usepackage{graphicx}
\usepackage{epstopdf}
\usepackage{algorithmic}
\usepackage{subcaption}

\ifpdf
  \DeclareGraphicsExtensions{.eps,.pdf,.png,.jpg}
\else
  \DeclareGraphicsExtensions{.eps}
\fi

\newsiamremark{remark}{Remark}
\newsiamremark{hypothesis}{Hypothesis}
\crefname{hypothesis}{Hypothesis}{Hypotheses}
\newsiamthm{claim}{Claim}
\newsiamremark{fact}{Fact}
\crefname{fact}{Fact}{Facts}
\newcommand{\average}[1]{\left\langle#1\right\rangle}
\def\BE{\mathrm{BE}}
\def\eps{\varepsilon}

\def\l{\langle}
\def\r{\rangle}
\def\p{\partial}
\def\d{\mathrm{d}}

\def\eps{\varepsilon}

\def\exp{\mathrm{exp}}

\headers{Reconstruct relaxation time}{P. Chen, I. M. Gamba, Q. Li and L. Wang}

\title{Reconstruction of heat relaxation index in phonon transport equation\thanks{Submitted to the editors \today.}}

\author{Peiyi Chen\thanks{Department of Mathematics, University of Wisconsin-Madison, Madison, WI 53706 USA 
  (\email{pchen345@wisc.edu}).}
\and Irene M. Gamba\thanks{Department of Mathematics and Oden Institute, University of Texas-Austin, Austin, TX 78712 USA
  (\email{gamba@math.utexas.edu}).}
\and Qin Li\thanks{Department of Mathematics, University of Wisconsin-Madison, Madison, WI 53706 USA 
  (\email{qinli@math.wisc.edu}).}
\and Li Wang\thanks{School of Mathematics, University of Minnesota, Minneapolis, MN 55455 USA 
  (\email{liwang@umn.edu}).}}

\usepackage{amsopn}

\begin{document}

\maketitle

\begin{abstract}
For nano-materials, heat conductivity is an ill-defined concept. This classical concept assumes the validity of Fourier's law, which states the heat flux is proportional to temperature gradient, with heat conductivity used to denote this ratio. However, this macroscopic constitutive relation breaks down at nano-scales. Instead, heat is propagated using phonon transport equation, an ab initio model derived from the first principle. In this equation, a material's thermal property is coded in a coefficient termed the relaxation time ($\tau$). We study an inverse problem in this paper, by using material's temperature response upon heat injection to infer the relaxation time. This inverse problem is formulated in a PDE-constrained optimization, and numerically solved by Stochastic Gradient Descent (SGD) method and its variants. In the execution of SGD, Fr\'echet derivative is computed and Lipschitz continuity is proved. This approach, in comparison to the earlier studies, honors the nano-structure of  of heat conductivity in a nano-material, and we numerically verify the break down of the Fourier's law. 
\end{abstract}



\section{Introduction}
How does heat propagate through materials at nanoscale? Heat conduction is of utmost importance in material science, particularly with the emergence of nano-technology~\cite{CGM-2002, DCCG-2006-accum, P-nano-2010, Regner-2013-Nature Comm}. However, at this scale where the material size is comparable to the phonon mean free path, the classical macroscopic constitutive law of heat conductance, namely the Fourier's law, is no longer valid, as reported by previous experiments~\cite{HLCM-2019-PhyB, Regner-2013-Nature Comm}. A more fundamental model that honors ab initio derivation should be in place for understanding heat conductance at this scale. For nano-materials, heat typically propagates via phonon and heat conduction is encoded in phonon dynamics. To this end, we deploy the phonon transport equation (PTE), a set of equations derived in~\cite{DCCG-2006-accum}, and then extensively investigated in material science literature~\cite{Chen-2001-BD, CHAM-2014-PhyB, Majumdar-1993-heat conduct}. This set of equation is not completely independent from our traditional study. In particular, in the diffusive regime, this phonon transport equation for nano-material is reduced to the classical heat equation through asymptotic expansion.

Several coefficients in the phonon transport equation are related to the thermal properties of the materials. This includes the relaxation time coefficient $(\tau)$, mean free path (MFP, $\Lambda$) distribution and reflection coefficient $(\eta)$. Studying the materials' thermal properties amounts to determining these coefficients. These coefficients cannot be measured directly, so instead, experimentally, one injects heat to the materials under study and observes its temperature response -- the so-called ``pump-probe" technique. The collected data is then used to complete the reconstruction. These studies are of huge interest in material science, see~\cite{Minnich-2012-Suppression, Koh-Cahill-2007, Chen-2015-MFP}. For instance, time domain thermoreflectance (TDTR) and frequency domain thermoreflectance (FDTR)~\cite{Cahill-2004-TDTR, Regner-2013-Nature Comm} are two experimental techniques employed to extract effective thermal conductivity. 

Among many coefficients, in this paper, we pay special attention to the phonon relaxation time, a quantity that carries significant role in determining materials' thermal property~\cite{Hadji-rlx-16,Hadji-rlx-18}. Typically the relaxation time is the amount of time for a material to return to its original state upon disturbance. In this context, the phonon relaxation time is the time of phonon at different frequencies needs to get reduced to the equilibrium states, and thus $\tau = \tau(\omega)$ is a function of frequency $\omega$. Experiments are usually conducted in a non-intrusive manner, so physicists can only measure the temperature response of a given bulk of material at the surface of this material. Using the paired data of injected heat and temperature response, to infer the phonon relaxation time poses an inverse problem.

We study the whole set of the problem. More specifically, we first examine the model and identify its ballistic and diffusive regime in section~\ref{sec: BTE model}. The classical heat equation will be derived in the limit scaling of large space and long time (section~\ref{sec: conductance}). This study allows us to identify the regime in which our problem is valid. This problem of reconstruction of $\tau$ has been extensively studied in materials science, with various approaches proposed. Many of these methods are derived under strong assumptions, but we nevertheless provide a summary of these approaches in section~\ref{sec:other_models}. We finally formulate our problem into an inverse problem setting, and cast it into a PDE-constrained optimization, see section~\ref{sec: inverse and alg}. To solve this optimization problem, gradient based solvers are deployed, including both Armijo line search and Adaptive Gradient method. We discuss our choice of Stochastic Gradient Descent (SGD) and the convergence of these two methods. To compute the gradient, the standard adjoint state method is used. Our perspective honors the nano-structure, and thus reflects the materials' microscopic information. The numerical results are presented at the end in section~\ref{sec: numerics}.

Upon linearization of the model equation in section~\ref{sec: BTE model}, we define several macroscopic quantities from the solution of linearized PTE. Especially, we define a heat conductance from PTE solution and show that it matches the bulk heat conductivity in the limit of large-space long-time scaling. 
In section~\ref{sec: gradient computation}, the adjoint system is introduced and the according Fr\'echet derivative is derived in detail. 
Our approach is practical and focus on the numerical properties, and we provide some numerical study of forward problem in section~\ref{sec: numerical forward eqn}. The trajectories of characteristics can be clearly seen in our ballistic regime, which gives an insight on how to set up the experimental measurements in such regime. We also study the numerical property of the loss function in section~\ref{sec: numerical loss fcn}, leading to a numerical observation of the sensitivity of loss function with respect to $\tau$ with the chosen source, measurement pair.

It is worth mentioning that phonon transport equation belongs to a large class of equations under the umbrella of kinetic theory. Though inverse problem associated with phonon transport equation is still at its infancy, inverse kinetic theory has been extensively studied in other contexts, especially for the radiative transfer equations (RTE). Radiative transfer equation is the set of equation that describes photon (light) propagation. Experimentally, light is injected into unknown media, and by using information of injected and reflected light intensity, one can infer the optical properties of the media. This technique has been extensively studied in remote sensing, medical imaging and atmospheric science~\cite{Ar-med-1999, CHS-2020-optic tomo}. See reviews~\cite{Ar-med-1999, AS-Optomo-2009, GB-trans-2009, Ren-med imag-2010}. Mathematically, this inference problem is also formulated as an inverse problem, with its mathematical structure extensively studied in~\cite{BJ-stable, Bal-average-2011, SU-opt tomo-2003}. In particular, well-posedness (unique and stable reconstruction) was studied using the singular decomposition technique was first developed and then strengthened in~\cite{CS-singular-1996, SU-opt tomo-2003}. In a more general setting,~\cite{LL-reconstruct-2020} discussed the unique reconstruction for the source term, and the extended results using kinetic tools were studied in~\cite{LS-2020-tool, LS-2022-unique}. 

Another related area of study is PDE-constrained optimization, which is both a classical numerical tool for solving inverse problems arising from PDEs and a topic of extensive research. In this approach, the objective function is formulated as the discrepancy between PDE-simulated data and true measurements. Optimization is then used to find the configuration of coefficients that best match the simulated data to the given measurements. This technique is most commonly applied to elliptic and hyperbolic PDEs~\cite{AS-Optomo-2009, AS-gradient-1998}. A key tool in this computation is the adjoint state method, which is used to compute the Fréchet derivative of the loss function. The loss function is typically a functional that maps the unknown coefficient (usually a function) to the mismatch value, so when using a gradient-based solver, the Fréchet derivative must be derived with knowledge of the underlying PDE model~\cite{HPUU-book-2009}. Both gradient-based and Hessian-based methods have been employed to numerically solve PDE-constrained optimization problems~\cite{SAD-newton-1993, TKAK-RTE-2011}. More recently, stochastic optimization methods have gained popularity for large-scale problems~\cite{BCN-SGD-2018, CLL-stochastic-2018, Duchi-AdaGrad-2011, JZZ-SGD-2020}.

\section{Boltzmann model for heat conductance} \label{sec: BTE model}
In this section, we present the PDE model for heat-conductance, and related preliminary calculation in the diffusion regime. In particular, we introduce the model and the linearization in Section~\ref{sec:phonon_model} and discuss the computation of physical quantities, such as the heat flux, temperature and heat conductance in Section~\ref{sec: conductance}.

\subsection{Phonon transport equation}\label{sec:phonon_model}
As discussed in the introduction, heat propagation is modeled at the fundamental level by a 1D phonon transport equation, a type of Boltzmann Equation (BTE) of BGK (Bhatnagar-Gross-Krook) form:
\begin{equation}\label{eqn: f-BTE-1}
\partial_t f + \mu v(\omega) \p_x f = - \dfrac{f-f_{\BE}(\cdot\,;\mathsf{T})}{\tau(\omega)}\,.
\end{equation}
Here $f(t,x,\mu,\omega)$ stands for the probability of finding a phonon particle at time $t$, space $x$, with phonon frequency $\omega\in[0,\infty)$ and velocity angle $\mu = \cos(\theta)\in[-1,1]$. In other words, $f\d\mu \d\omega$ represents the number of phonons at time $t$ and location $x$. Two parameters are crucial in this equation: $v(\omega)$ is phonon group velocity and $\tau(\omega)$ is phonon relaxation time. They are determined by the material under study. The left hand side of the equation describes the phonon transport with velocity $v(\omega)\mu$, and the right hand side is the BGK operator, in charge of driving $f$ to $f_{\BE}(\cdot\,;\mathsf{T})$, the Bose-Einstein distribution. This  distribution takes the form of:
\begin{equation}\label{eqn: Bose-Einstein dist}
f_{\BE}(\cdot\,;\mathsf{T}) = [\exp{(\hbar\omega/k_B \mathsf{T})}-1]^{-1}\,,
\end{equation}
where the temperature $\mathsf{T}=\mathsf{T}(t,x)$ is determined by the conservation of phonon energy: 
\begin{equation}\label{eqn: T-conserve}
\int \hbar\omega D(\omega)\dfrac{f-f_{\BE}(\cdot\,;\mathsf{T})}{\tau(\omega)}\,\d\mu \d\omega =0\,.
\end{equation}
Here $D(\omega)$ is the phonon density of states and $\hbar\omega$ is the unit energy.

Throughout the paper, we use the notation:
\[
\langle f \rangle_{\mu,\omega} = \frac{1}{C} \int_0^{\infty}\int_{-1}^1 f\,\d\mu\d\omega \,, \quad \mathrm{where} \ C = \int_{0}^{\infty} \!\int_{-1}^{1} 1\, \d\mu \d \omega\,.
\]
Here $C$ serves as the normalization constant.

Most experiments are conducted close to the room temperature $T_0$, and it is typically sufficient to examine the equation linearized around $T_0$. To do so, we define the deviational energy distribution function:
\begin{equation*}
g(t,x,\mu,\omega) = \hbar\omega D(\omega)[f - f_{\BE}(\cdot\,;T_0)]\,.
\end{equation*}

Then \eqref{eqn: f-BTE-1} rewrites into:
\begin{equation}\label{eqn: g-BTE-2}
\p_t g + \mu v(\omega) \p_x g = - \dfrac{g - \hbar\omega D(\omega)(f_{\BE}(\cdot\,; \mathsf{T})-f_{\BE}(\cdot\,;T_0))}{\tau(\omega)}\,.
\end{equation}

Since the perturbation of temperature is small, we can linearize ~\eqref{eqn: f-BTE-1}, as outlined in \cite{HCRM-2017-PhyB}. In particular, denote $T$ as the temperature difference: $T = \mathsf{T}-T_0$, then the equilibrium can be simplified:
\begin{equation*}
f_{\BE}(\cdot\,; \mathsf{T}) \approx f_{\BE}(\cdot\,;T_0) + \left.\dfrac{\p f_{\BE}}{\p \mathsf{T}}\right\vert_{T_0} T\,\,.
\end{equation*}
Then~\eqref{eqn: g-BTE-2} is approximated by:
\begin{equation}\label{eqn: g-BTE-3}
\p_t g + \mu v(\omega) \p_x g = - \dfrac{g - g^*(\omega)T}{\tau(\omega)}\,,
\end{equation}
where $g^*(\omega)$ is the equilibrium for the linearized system:
\[
g^*(\omega):=\hbar \omega D(\omega)\left.\frac{\p f_{\BE}}{\p \mathsf{T}}\right\vert_{T_0}\,.
\]
To ensure the energy conservation in \eqref{eqn: g-BTE-3}, we require $T(t,x)$ to have the form: 
\begin{equation}\label{energy-conserv-temp}
\average{\frac{g(t,x,\mu,\omega)}{\tau(\omega)} - \frac{g^*(\omega)}{\tau(\omega)}T}_{\mu\,,\omega} = 0\,,\quad\text{i.e., }\quad 
T(t,x) = \frac{\l g/\tau\r_{\mu,\omega}}{\l g^*/\tau\r_{\mu,\omega}}\,.
\end{equation}
Since $T$ is determined by $T_0$ and $\mathsf{T}$, it seems no freedom to require $T$ to have this form. Nevertheless, from~\eqref{eqn: T-conserve}, we can derive~\eqref{energy-conserv-temp} if $g^*$ is from the exact difference between $f_{\BE}(\cdot;T)$ and $f_{\BE}(\cdot;T_0)$.

Several macroscopic quantities associated with this equation deserves our attention. In particular,
\begin{itemize}
    \item heat flux:
    \begin{equation}\label{eqn: flux}
    q(t,x) := \average{\mu v(\omega) g}_{\mu,\omega} \,;
    \end{equation}
    \item temperature (recall~\eqref{energy-conserv-temp}), we have:
    \begin{equation}\label{eqn: temperature}
    T(t,x) := \frac{\average{g/\tau}_{\mu,\omega}}{\average{g^*/\tau}_{\mu,\omega}}\,;
    \end{equation}
    \item and finally, inspired by the Fourier's law for the heat equation, we define heat conductance from the PTE solution~\eqref{eqn: g-BTE-3}:
    \begin{equation}\label{eqn: conductance}
    \kappa(t,x) := - \frac{q}{\p_x T} = - \frac{\average{\mu v g}_{\mu,\omega}}{\p_x \average{g/\tau}_{\mu,\omega}} \average{g^*/\tau}_{\mu,\omega}\,.
    \end{equation}
\end{itemize}

Equation~\eqref{eqn: g-BTE-3} is equipped with initial and boundary condition. To cope with the experimental setup,  we assume a Dirichlet incoming boundary condition at $x=0$ and a reflective boundary condition at $x=1$. Summarizing the system:
\begin{align} \label{eqn: g-BTE-IBVP} \left\{\begin{array}{ll}
\p_t g + \mu v(\omega) \p_x g = \dfrac{1}{\tau(\omega)}\mathcal{L}_0[g],&\ 0<x<1, \\[3mm]
g(t=0,x,\mu,\omega) = 0,\\[3mm]
g(t\not=0,x=0,\mu,\omega) = \phi(t, \mu,\omega),& \ \text{for}\ \mu>0, \\[3mm]
g(t,x=1,-\mu,\omega) = g(t,x=1,\mu,\omega), & \ \text{for}\ \mu>0\,.
\end{array}\right.
\end{align}
Here the collision operator on the right hand sides is:
\begin{equation*} 
\mathcal{L}_0[g](t,x,\mu,\omega) := T g^*(\omega) - g = \left(\frac{\l g/\tau\r_{\mu,\omega}}{\l g^*/\tau\r_{\omega}} g^* - g\right)\,,
\end{equation*}
and $\phi$ is the incoming source. In experiments, it is typical to choose $\phi$ to be periodic in time~\cite{Collins-2013, Regner-2013-Nature Comm}, and solution to the phonon transport equation~\eqref{eqn: g-BTE-2} has been written in Fourier variables in~\cite{HLCM-2019-PhyB}. In our paper, to visualize the mathematical structure of the solution, we assume $\phi$ taking  a form of a narrow Gaussian concentrated at $t_0, \mu_0$ and $\omega_0$ (to-be-tuned) with variances $\eps_i, i=1,2,3$, respectively. Note that the solution space is linear since the equation is linear, and therefore we have the freedom to choose different basis. In particular, we parameterize this space using basis functions of the form:
\begin{equation}\label{eqn:phi_cond}
\phi(t,\mu,\omega) = \phi_{\eps_1}(t-t_0) \phi_{\eps_2}(\mu-\mu_0)\phi_{\eps_3}(\omega-\omega_0)
\end{equation}
instead of the Fourier basis, and this won't introduce any inconsistency with the experiment.

For computational convenience, we define a new quantity:
\begin{equation} \label{eqn: def of h*}
h(t,x,\mu,\omega) := \frac{1}{\tau(\omega)}g(t,x,\mu,\omega)\,,\quad  h^*(\omega) := \frac{1}{\tau(\omega)}g^*(\omega)\,.
\end{equation}
Note that $h$ does not carry any physical meanings, but will make the calculation significantly easier. This quantity satisfies:
\begin{equation} \label{eqn: h-BTE-IBVP}
\left\{\begin{array}{ll}
\partial_t h + \mu v(\omega) \partial_x h = \dfrac{1}{\tau(\omega)} \mathcal{L}[h] \,, \quad & 0<x<1 \\[3mm]
h(t=0, x,\mu,\omega) = 0\,, \\[3mm]
h(t\not=0, x=0,\mu,\omega) = \dfrac{\phi}{\tau}\,, & \mu>0 \\[3mm]
h(t, x=1,-\mu,\omega) = h(t, x=1,\mu,\omega)\,, & \mu>0\,,
\end{array}\right.\,
\end{equation}
where the collision operator $\mathcal L $ takes the form
\begin{equation} \label{LL}
 \mathcal L [h] :=\frac{\l h \r_{\mu,\omega}}{\l h^* \r_{\omega}} h^* - h   \,.
\end{equation}

\begin{proposition} \label{propL}
Here we summarize the properties of $\mathcal L$ in \eqref{LL}.
\begin{enumerate}
\item $\mathcal{L}[h]$ is linear in $h$.
\item Conservation: 
\begin{equation*}
\average{\mathcal{L}[h]}_{\mu,\omega} = \frac{\l h \r_{\mu,\omega}}{\l h^* \r_{\omega}} \l h^* \r_{\omega} - \l h\r_{\mu,\omega} = 0.
\end{equation*}
\item $\mathcal L$ has a one-dimensional kernel:
\[
\text{Ker}\ \mathcal{L}=\{h(t, x, \mu, \omega) = C(t, x)h^*(\omega),\ \text{constant in}\ \mu.\}
\]
\item $\mathcal{L}$ is self-adjoint with weight $\frac{1}{h^*}$.
\[
\average{\mathcal{L}[h]\,, p\, \dfrac{1}{h^*}}_{\mu,\omega}=\average{h \,,\mathcal{L}[p]\, \dfrac{1}{h^*}}_{\mu,\omega}
\]
\item For any $g\in\text{Range}\ \mathcal{L}$, there is a unique $h\in(\text{Ker}\ \mathcal{L})^\perp$ so that $\mathcal{L}^{-1}[g]=h+\text{Ker}\ \mathcal{L}$. Here $(\text{Ker}\ \mathcal{L})^\perp$ is the $L_2$ perpendicular space weighted by $\frac{1}{h^*}$.
\end{enumerate}
\end{proposition}
All properties are straightforward. In particular, for showing $4$:
\begin{equation*}\begin{aligned}
\average{\mathcal{L}[h] p\, \dfrac{1}{h^*}}_{\mu,\omega} = \average{\left(\dfrac{\l h\r_{\mu,\omega}}{\l h^*\r_{\omega}} h^* - h\right) p\, \frac{1}{h^*} }_{\mu,\omega} = \int_{\mu,\omega} h \left(\dfrac{\l p\r_{\mu,\omega}}{\l h^*\r_{\mu,\omega}} h^* - p\right)\, \frac{1}{h^*} \,\d\omega \d\mu.
\end{aligned}\end{equation*}

\subsection{Parabolic limit in the diffusion regime} 
\label{sec: conductance}
When time and space are rescaled properly, the phonon transport equation presents diffusive features, and the equation asymptotically becomes the diffusion equation. On the macroscopic level, one can further deduce the classical Fourier's law, meaning the heat flux is a constant proportion of the temperature gradient. Mathematically, it means $\kappa$, as defined in~\eqref{eqn: conductance} becomes a constant.
\subsubsection{Diffusion equation derivation} 
It is typical that in~\eqref{eqn: g-BTE-3}, relaxation time of phonon is the order of picoseconds, i.e.  $10^{-12}-10^{-11} \mathrm{s}$, and mean free path is the order of nanometers, i.e. $10^{-9}-10^{-8} \mathrm{m}$.

Upon utilizing nondimensionalization, we introduce
\begin{equation*}
\hat{t} = \frac{t}{\delta}\,,\ \hat{x} = \frac{x}{\varepsilon}\,,
\end{equation*}
and in the rescaled variables, 
\begin{equation*}
\hat{g}(t,x,\mu,\omega) := g(\hat{t},\hat{x},\mu,\omega)=g\left(\frac{t}{\delta}, \frac{x}{\varepsilon}, \mu,\omega\right)\,.
\end{equation*}
If the experiment is conducted in nanoseconds and micrometers, $\delta=O(10^{-3}), \varepsilon=O(10^{-3})$ and roughly they are equal. We call this the quasi-ballistic regime. On the other hand, if the experimental observation is in megahertz (microseconds) and micrometers respectively, compatible with the TDTR experiments, $\delta=O(10^{-6})$, $\varepsilon=O(10^{-3})$, and roughly $\delta = \varepsilon^2$. We call this the diffusive regime.

In the diffusive regime where $\delta=\varepsilon^2$, $\hat{g}$ solves:
\begin{equation*}
\varepsilon^2 \partial_t \hat{g}_\varepsilon + \varepsilon \mu v(\omega) \partial_x \hat{g}_\varepsilon = \frac{1}{\tau(\omega)}\mathcal{L}_0[\hat{g}_\eps]\,.
\end{equation*}
We drop hats and equivalently, $g_\varepsilon$ solves:
\begin{equation}\label{eqn: g-diff-regime}
\p_t g_\eps + \frac{1}{\eps} \mu v \p_x g_\eps = \frac{1}{\eps^2} \dfrac{1}{\tau(\omega)}\mathcal{L}_0[g_\eps]\,.
\end{equation}

The following proposition characterizes the evolution of the macroscopic quantity when $\varepsilon\to0$, that is, in the diffusive regime:
\begin{proposition}\label{prop: limit eqn}
When $\eps\to 0$, the solution to \eqref{eqn: g-diff-regime} converges to the solution of heat equation in the sense that $\average{g_\eps}_{\mu,\omega} \to \average{g^*}_{\mu,\omega} u(t,x)$,
with $u(t,x)$, the temperature, solving:
\begin{equation} \label{eqn: heat}
\p_t u - \frac{\kappa}{\average{g^*}_{\mu,\omega}} \p_{xx}u = 0\,, \qquad\text{with}\quad
\kappa :=\frac{1}{3} \average{\tau v^2 g^*}_\omega\,,
\end{equation}
where $\kappa$ is the bulk heat conductivity. Furthermore, recall the definition of heat flux~\eqref{eqn: flux},
$\lim_{\eps\to0}q_\eps(t,x) = - \frac{1}{3}\average{v^2 \tau g^*}_\omega \p_x u $ and $\lim_{\eps\to0}
T_\eps(t,x)= u(t,x)$, making, according to~\eqref{eqn: conductance}:
\begin{equation}\label{eqn:conv_kappa}
\kappa_\varepsilon \rightarrow \kappa\,.
\end{equation}
\end{proposition}

We provide the proof in Supplementary Material~SM1. The original rigorous proof with validation of asymptotic expansion is in~\cite{BSS-diffusive-1984}. The significance of the proposition is immediate. It connects the unfamiliar phonon transport equation~\eqref{eqn: g-diff-regime} with the heat equation, a classical model. Typically, by energy conservation, one has $\partial_tQ = \partial_xq$ with $Q\propto T$ representing heat energy per unit volume and $q$ being the heat flux. Assuming the Fourier's law holds: $q=-\kappa \partial_x T$, we land with the same equation as~\eqref{eqn: heat}. This proposition suggests that this holds only in the $\eps\to0$ limit where both spatial and temporal scalings are large, and the heat equation stops being valid for large $\eps$.

Our simulation of phonon transport confirms the finding. We show in Fig~\ref{fig: ballistic vs. diffusive (3 graphs)} with $\eps=1$ and $\eps=0.1$ respectively using boundary and initial condition described in~\eqref{eqn:phi_cond}. As seen in Figure~\ref{fig: ballistic vs. diffusive (3 graphs)}(A)-(C) for $\eps=1$, the solution demonstrates strong ballistic behavior with clear characteristic lines. In this case, the heat conductivity is non uniform in the $(t,x)$ domain. In particular, along the characteristic, as shown in Figure~\ref{fig: ballistic vs. diffusive (3 graphs)}(C), the heat conductivity computed from~\eqref{eqn: conductance} is no longer a constant, but varies dramatically. The ray goes through a reflection at $x=1$ and generates a strong singularity at $t=t_r$, the arrival time of the ray. When $\eps=0.1$, the equation falls in the diffusion regime, as shown in Figure~\ref{fig: ballistic vs. diffusive (3 graphs)}(D)-(E) the solution diffuses out, and the heat conductivity becomes a constant quickly (around $t=0.125$), and settles with a value approximately $0.046$, and is the same value computed by~\eqref{eqn: heat}.

\begin{figure}[htbp]
     \centering
     \begin{subfigure}[b]{0.32\textwidth}
         \centering
         \includegraphics[width=\textwidth]{ 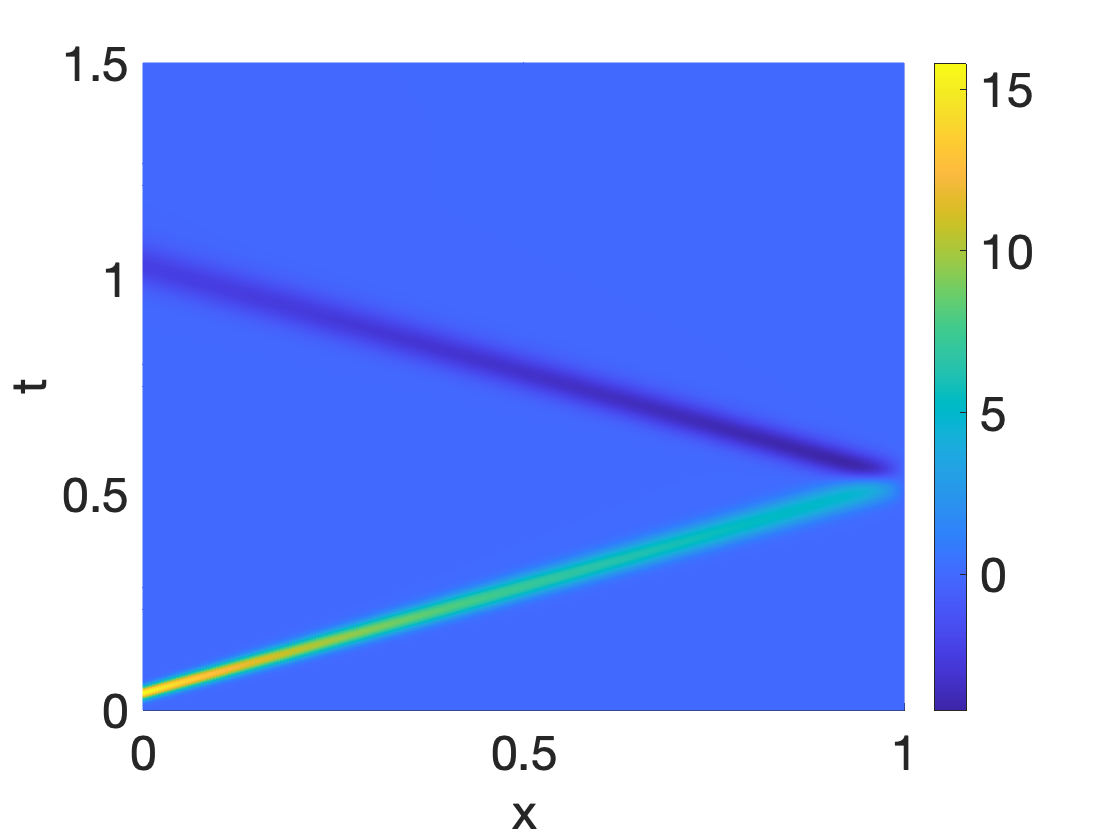}
         \caption{$q\,, \varepsilon=1$}
         \label{fig: flux-1}
     \end{subfigure}
     \hfill
     \begin{subfigure}[b]{0.32\textwidth}
         \centering
         \includegraphics[width=\textwidth]{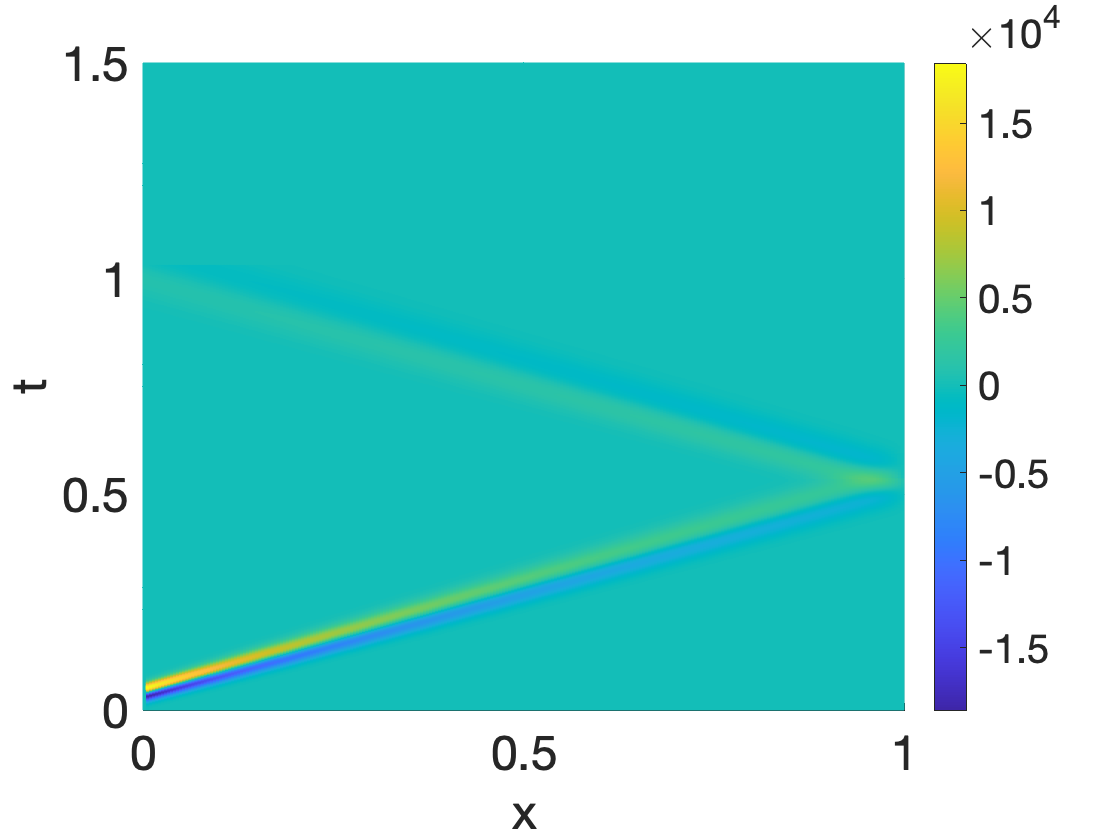}
         \caption{$\nabla_x T\,, \varepsilon=1$}
         \label{fig: dT-1}
     \end{subfigure}
     \hfill
     \begin{subfigure}[b]{0.32\textwidth}
         \centering
         \includegraphics[width=\textwidth]{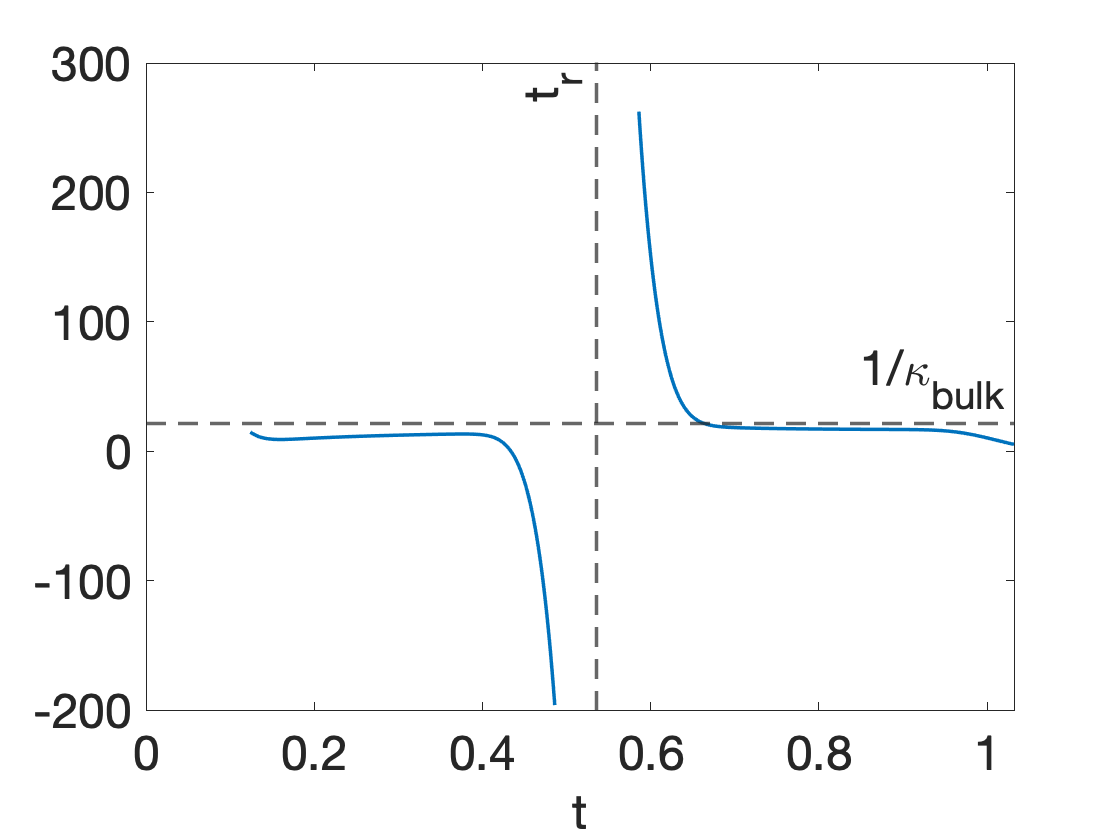}
         \caption{$1/\kappa\,, \varepsilon=1$}
         \label{fig: kappa-1}
     \end{subfigure}
     \\
     \begin{subfigure}[b]{0.32\textwidth}
         \centering
         \includegraphics[width=\textwidth]{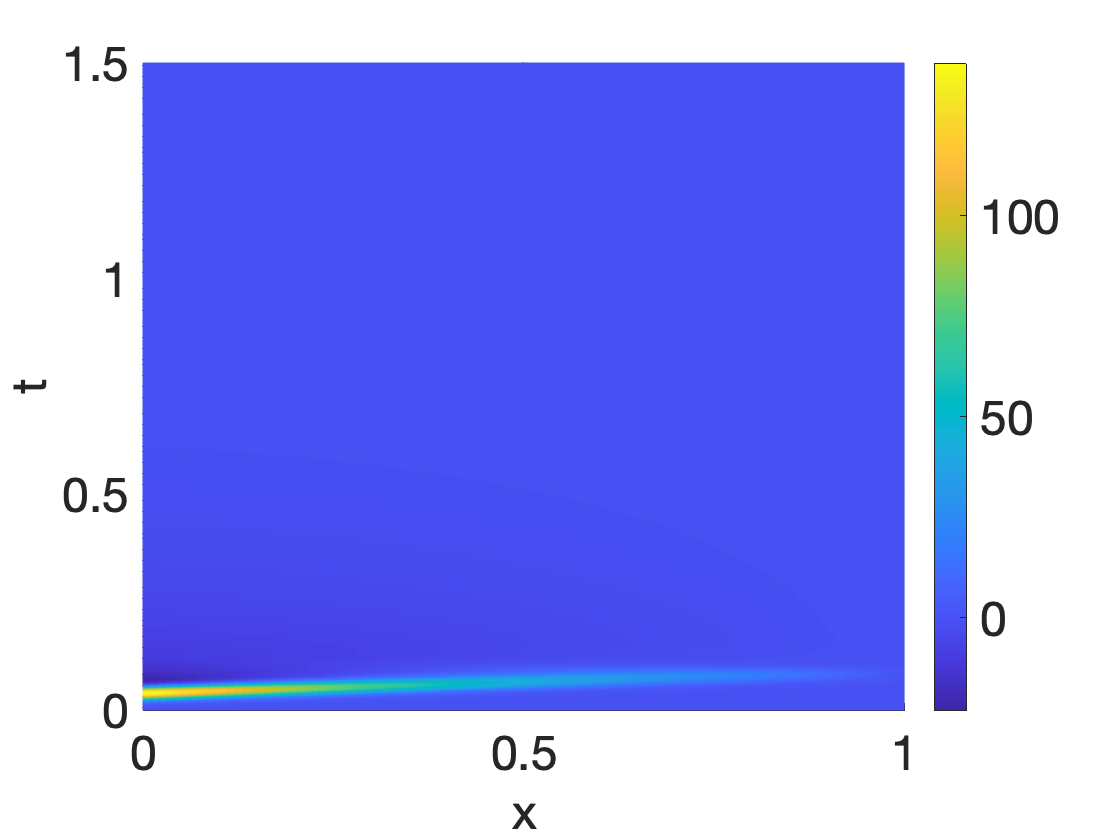}
         \caption{$q\,, \varepsilon=0.1$}
         \label{fig: flux-2}
     \end{subfigure}
     \hfill
     \begin{subfigure}[b]{0.32\textwidth}
         \centering
         \includegraphics[width=\textwidth]{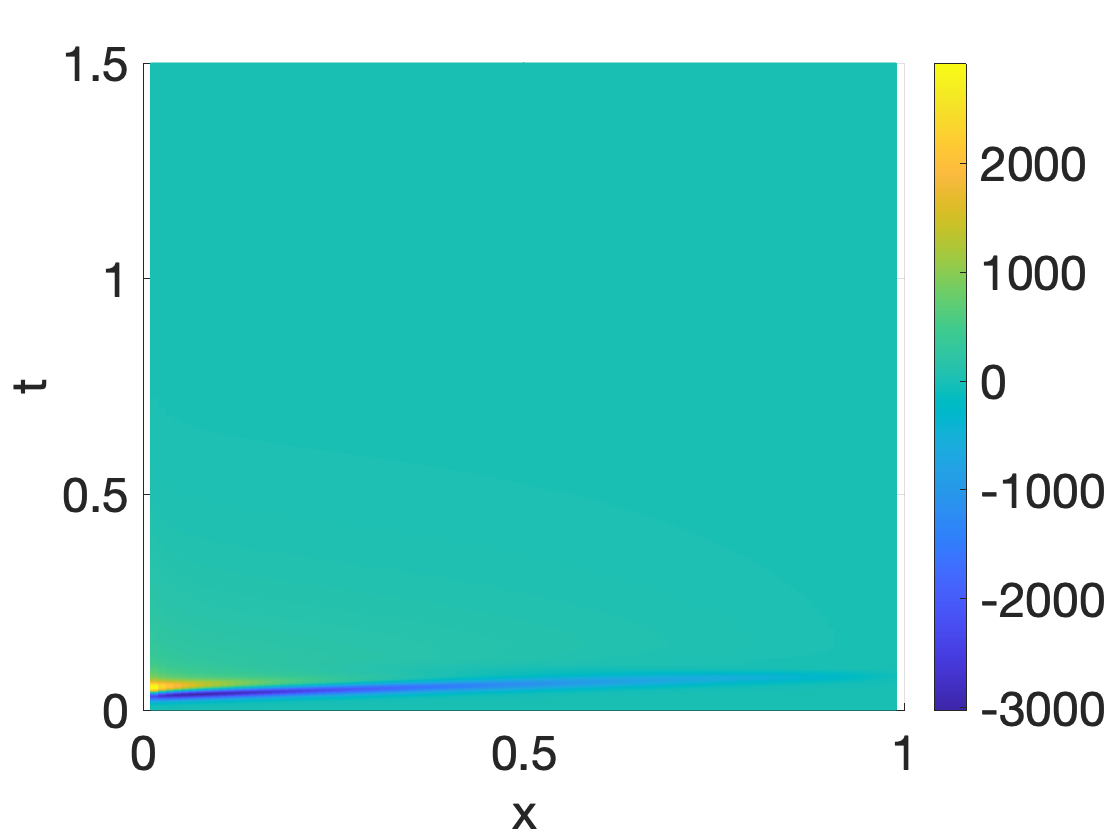}
         \caption{$\nabla_x T\,, \varepsilon=0.1$}
         \label{fig: dT-2}
     \end{subfigure}
     \hfill
     \begin{subfigure}[b]{0.32\textwidth}
         \centering
         \includegraphics[width=\textwidth]{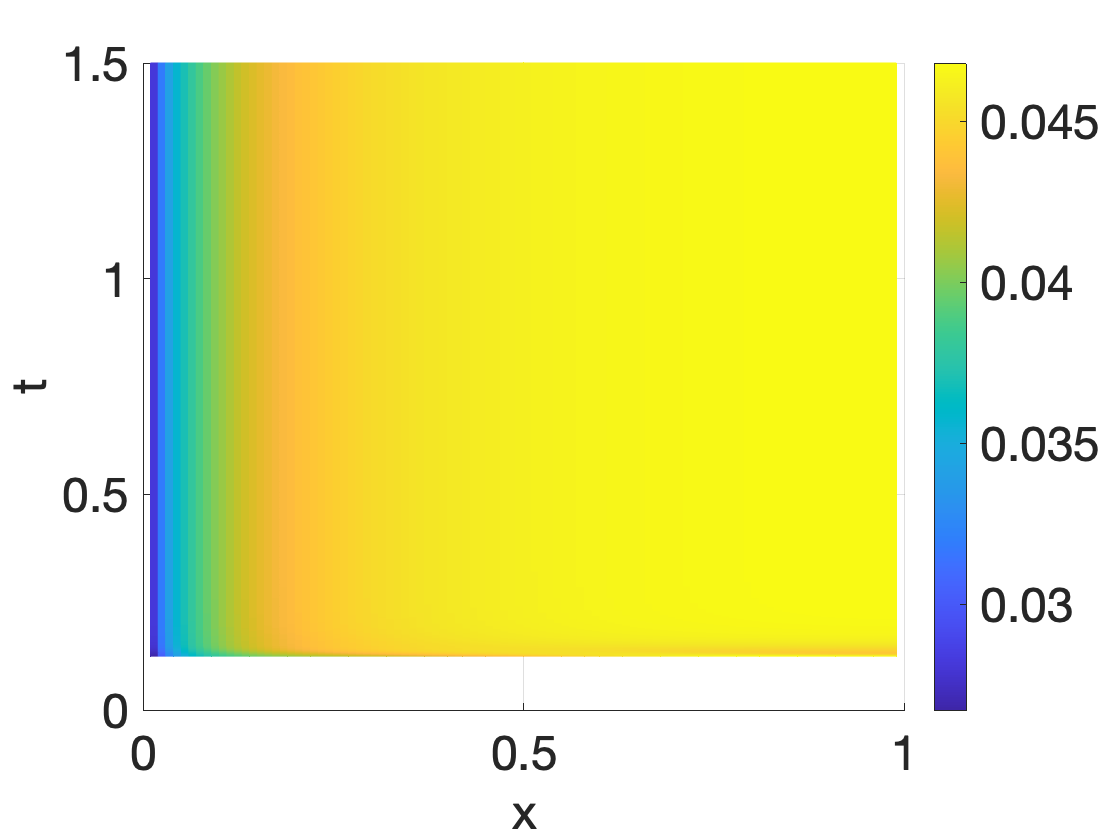}
         \caption{$\kappa\,, \varepsilon=0.1$}
         \label{fig: kappa-2}
     \end{subfigure}
        \caption{Comparison between quasi-ballistic regime and diffusive regime}
        \label{fig: ballistic vs. diffusive (3 graphs)}
\end{figure}

\section{Existing Models Relating Heat Conductance to Phonon Frequencies} \label{sec:other_models}
Since the heat equation no longer holds, and the Fourier's law breaks down when materials shrink in size, more refined models are needed to describe the nano-scale heat conductance.

However, simulating PTE directly is a complex task. Therefore, various intermediate models have been proposed to provide corrections to the definition of ``effective thermal conductivity" by exploring the more detailed relationship between thermal conductivity and phonon frequency. These models are simpler than PTE but sacrifice a portion of the accuracy inherent in PTE. We summarize these models below.

\vspace{0.1in}
\noindent\textbf{Grey approximation.}
The first step towards addressing the non-equilibrium behavior of the phonon distribution is by Debye, who used the grey approximation~\cite{Kittel-book-solid}. This is to assume that group velocity $v_G$ and relaxation time $\tau_G$ does not depend on phonon frequency $\omega$. As a result, the heat conductance $\kappa$ obtained in the Proposition~\ref{prop: limit eqn} simplifies to
\begin{equation} \label{eqn: grey}
\kappa = \int_0^\infty \dfrac{1}{3} g^*(\omega) v_G^2 \tau_G \,\d\omega = \dfrac{1}{3}C v_G^2 \tau_G\,, \quad\text{with}\quad C = \average{g^*(\omega)}_\omega\,.
\end{equation}
Here $C$ represents volumetric heat capacity.

It has been observed that assuming independence of group velocity and relaxation time on $\omega$ is a significant simplification. In certain experimental settings \cite{Regner-2013-Nature Comm}, it has been shown that such assumption does not hold, which renders the grey approximation invalid.

\vspace{0.1in}
\noindent \textbf{Accumulation formula.}
To emphasize the role of different phonon frequencies in heat conductance, authors in~\cite{DCCG-2006-accum} proposed a concept named ``the thermal conductivity accumulation function" defined as
\begin{equation} \label{eqn: acuumulation model}
\kappa_{\text{accum}}(\Lambda^*) = \int_{\omega_{\min}}^{\omega_{\max}} \dfrac{1}{3} C(\omega)v(\omega)^2\tau(\omega) \,\d\omega\,,
\end{equation}
where $C(\omega)$ is the volumetric specific heat capacity as a fuction of frequency $\omega$. It was then argued that only phonons within a certain frequency range (or mean free path) contribute to the accumulation of heat conductivity. This range was denoted as $(\omega_{\min}, \omega_{\max})$. 

In practice, \eqref{eqn: acuumulation model} is computed via integration with respect to the mean free path $\Lambda(\omega) := v(\omega)\tau(\omega)$, i.e., 
\begin{equation}
\kappa_{\text{accum}}(\Lambda^*) = \int_0^{\Lambda^*} \dfrac{1}{3} C(\Lambda)v(\Lambda)\Lambda \left(\frac{d\Lambda}{d\omega}\right)^{-1} \, \d\Lambda\,,
\end{equation}
where $\Lambda^*$ is the cut-off length and is subject to debate. In~\cite{Chen-2015-MFP}, the authors used the heater size, i.e. the width of metal lines on top of the substrate, as a reference for this cut-off. In~\cite{Koh-Cahill-2007}, the authors proposed to set the cut-off as the ``the penetration depth",  defined by $L_p := \sqrt{\kappa/C\pi f}$, where $f$, $\kappa$ and $C$ are heating frequency, reference heat conductivity and heat capacity respectively. This model was then deployed in experiment~\cite{Regner-2013-Nature Comm}, where, by  varying heating frequency $f$ and fitting the conductivity empirically, one can extract  the distribution of $\Lambda$.

\vspace{0.1in}
\noindent \textbf{Suppression model.} 
Instead of directly truncating the mean free paths (MFPs) as in Eq. \eqref{eqn: acuumulation model}, authors in \cite{Minnich-2012-Suppression} propose an alternative approach to determine the distribution function of the mean free path. Ideally, this function will naturally take a very small value when the mean free path falls in a region that does not contribute much to the thermal conductivity.

In particular, they introduce a heat flux suppression function $S(\Lambda)$, which represents the reduction in phonon heat flux relative to the prediction from Fourier's law.  Then, the effective thermal conductivity can be obtained by
\begin{equation}\label{eqn: suppression model}
\kappa(L) = \int_0^\infty S(\Lambda/L)f(\Lambda) \,\d\Lambda\,.
\end{equation}
Here, $L$ is a tunable characteristic thermal length in experiments, such as the grating wavelength in transient grating geometry. By varying this value, different values of $\kappa(L)$ can be measured. $f(\Lambda)$ stands for the phonon MFP density,  which is the interest of reconstruction. Since $S(\Lambda/L)$ is specified for some specific heating geometry, $f(\Lambda)$ can be inferred.

\vspace{0.1in}
All these models integrate more details to correct the heat equation. Furthermore, these formulas provide an easy pathway to infer phonon properties. More specifically, by tuning the size of the metal heater or the frequency of the pump beam in metal grating experiments~\cite{Chen-2015-MFP} or time-domain thermo-reflectance (TDTR) experiments~\cite{Cahill-2004-TDTR, Regner-2013-Nature Comm}, the thermal conductivity $\kappa$ in~\eqref{eqn: acuumulation model} or~\eqref{eqn: suppression model} can be measured. One can then use it to infer phonon relaxation time and the distribution of phonon mean free path using these formulas.

However, noting the similarity of the formula (\eqref{eqn: grey}, \eqref{eqn: acuumulation model} or \eqref{eqn: suppression model}) to~\eqref{eqn: heat}, the models nevertheless fundamentally rely on the Fourier law assumption and lack the full-scope information presented in PTE. We propose to revisit the reconstruction starting from the most fundamental model, devoid of any additional assumptions. The temperature response of the material is measured to implicitly infer the relaxation time $\tau$ through the PDE-constrained inversion. A similar approach was taken in~\cite{Hadji-rlx-16}, in which the authors deployed a MC-type Bayesian optimization solver to directly reconstruct $\tau$.

\section{Inverse problem setup and PDE-constrained optimization} \label{sec: inverse and alg}
Our goal is to recover phonon relaxation time $\tau$ from macroscopic temperature measurements. We present the setup, the associated PDE-constrained optimization problem, and the related algorithm. In the execution of the optimization, functional gradient (Fr\'echet derivative) is needed and we also provide the computation. Throughout the section, we set $\varepsilon=1$ and for the reconstruction, the formulation presented through $h$ gives the best simplicity~\eqref{eqn: h-BTE-IBVP}, and is thus the equation that we rely on.

\subsection{Inverse problem setup}\label{subsec: inverse setup}
It is typical that in the lab setting, only macroscopic quantities are measurable, and they have to be measured in a non-intrusive manner. In particular, one can adjust the input heat source $\phi$~\eqref{eqn:phi_cond} and measure the resulting temperature change $T(t)$ only at the surface $x=0$. Denoting $\psi(t)$ the window function for the measurement conducted in time, we can define the forward map:
\begin{equation*}\label{eqn: forward mapping}
\Lambda_{\tau}: (\phi,\psi) \mapsto \l T(x=0,t), \psi(t) \r_t\,,\quad\text{where}\quad T(t,x) = \frac{\average{h}_{\mu,\omega}}{\average{h^*}_{\mu,\omega}}\,.
\end{equation*}
It is typical to choose $\psi(t)$ as to concentrate on the measurement time $t_R$, with variance $\eps_4$:
\begin{equation}\label{eqn: psi_cond}
\psi(t) = \phi_{\eps_4}(t-t_R)\,.
\end{equation}

For a given $\tau(\omega)$, PDE is fully determined, and when equipped with boundary condition $\phi$, produces a unique solution and thus generates a unique evaluation of $\Lambda_{\tau}(\phi,\psi)$. The inverse problem is to utilize many pairs of $\Lambda_{\tau}(\phi_i,\psi_i)$ to reconstruct $\tau(\omega)$. It is important to note that though $\Lambda_\tau$ is a linear map of $\phi$ and $\psi$, it has a non-linear dependence in $\tau$, so the inverse problem in nonlinear.

Denote $\tau^\ast$ the groundtruth relaxation time. We look to reconstruct it using the available data:
\begin{equation}\label{eqn:real_data}
\mathrm{d}_i=\Lambda_{\tau^\ast}(\phi_i,\psi_i)\,, \quad i=1,\cdots,N\,.
\end{equation}

Numerically, a PDE-constrained optimization can be formulated to achieve the task. In particular, we look for the function $\tau(\omega)$ so that its PDE-simulated data $\Lambda_{\tau}(\phi_i,\psi_i)$ agree as much as possible with $\{\d_i\}$.

To be more specific, define the mismatch function
\begin{equation}\label{eqn: mismatch}
l_i := \Lambda_\tau(\phi_i,\psi_i) - \Lambda_{\tau^*}(\phi_i,\psi_i) = \left\l \dfrac{\average{h}_{x=0,\mu,\omega}}{\average{ h^*}_{\mu,\omega}}, \psi_i \right\r_t - \mathrm{d}_i\,,
\end{equation}
then the PDE constrained optimization problem is stated as follows:
\begin{equation}\label{eqn: PDE-constrained opt}
\begin{array}{ll}
\text{min}_{\tau(\omega)} & L[\tau] := \frac{1}{N} \sum_{i=1}^N L_i\,,\quad \text{with}\quad L_i =\frac{1}{2}  l_i^2\,, \\[3mm]
\text{s.t.}  & 
\left\{\begin{array}{ll}
\partial_t h + \mu v(\omega) \partial_x h = \dfrac{1}{\tau(\omega)} \mathcal{L}[h]\,, \quad & 0<x<1 \\[3mm]
h(t=0, x,\mu,\omega) = 0\,, \\[3mm]
h(t\not=0, x=0,\mu,\omega) = \dfrac{\phi_i}{\tau}\,, & \mu>0\,, \\[3mm]
h(t, x=1,-\mu,\omega) = h(t, x=1,\mu,\omega)\,, & \mu>0\,.
\end{array}\right.
\end{array}
\end{equation}
The loss function $L$ is the average of $N$ sub-loss functions, with $L_i$ measuring the mismatch between the PDE-generated solution and the $i$-th data point. Each $L_i$ is a functional of phonon relaxation time $\tau(\omega)$. It is our goal to find a $\hat{\tau}(\omega)$ that minimizes $L$ under the constraint from the PDE system. 

In practice, this optimization is achieved through either optimize-then-discretize (OTD) or discretize-then-optimize (DTO) approaches. Here, we choose the former approach, initially formulating the optimization algorithm at the continuous level, then approximating it numerically using PDE solvers.

\subsection{Optimization algorithm}\label{sec: SGD algs}
There are many choices of algorithms for implementing optimization. Broadly, one must decide between gradient-based solvers or Hessian-based solvers. In our scenario, representing $\tau(\omega)$ as a numerical vector involves very high-dimensional data. Consequently, computing the Hessian becomes computationally intensive. Therefore, throughout our implementation, we rely on gradient-based solvers:
\begin{equation}\label{eqn: gradient update}
\tau^{n+1} = \tau^{n} + F\left(\dfrac{\delta L}{\delta \tau}[\tau^1]\,,\cdots\dfrac{\delta L}{\delta \tau}[\tau^n]\right)\,.
\end{equation}
The formula suggests that we update the evaluation of $\tau$ by incorporating the gradient information from  past trajectories. Among various gradient-based methods, we opt for Stochastic Gradient Descent (SGD) and utilize both the Armijo method and the Adaptive Gradient (AdaGrad) method for estimating step sizes.

SGD is a variant of gradient descent method and becomes immensely popular in machine learning~\cite{BCN-SGD-2018}. It is designed to tackle optimization problems of the form $f(x) = \frac{1}{N}\sum_i f_i(x)$. Its approach involves replacing $\nabla f$ in each iteration with $\nabla f_\xi$, where $\xi$ is randomly drawn from $1, \cdots, N$. This process ensures that $\nabla f$ is recovered at the expectation level: $\nabla f = \mathbb{E}[\nabla f_\xi]$. In other words the gradient chosen by the SGD on average recovers the vanilla gradient. However, per iteration, the computational cost for SGD is significantly reduced since it only calls for one gradient computation per iteration. The theory behind SGD has been extensively developed in the literature, and we will only reference one key result on convergence in~\cite{BCN-SGD-2018}.
\begin{theorem}\label{thm: SGD converge}
Assuming that the gradient is obtained unbiasedly:
\[
\mathbb{E}_\xi[\nabla f_\xi(\tau)] = \nabla f(\tau)\,,
\]
then, under certain assumptions on the variants of $\nabla f_i$, we achieve the convergence of SGD. Namely, let
\begin{equation} \label{0421}
\tau^{k+1}=\tau^k-\alpha_k \nabla f_{\xi_n}(\tau^k)\,,
\end{equation}
then with properly tuned $\alpha_k$:
\[\liminf_{k \rightarrow \infty} \,\mathbb{E}[\|\nabla f(\tau^k)\|_2^2] = 0 \,.\]
\end{theorem}

This strategy aligns well with our objectives, since our loss function~\eqref{eqn: PDE-constrained opt} takes the form $L=\frac{1}{N}\sum_iL_i$, enabling direct application of SGD. Specifically, at each iteration, we sample a random variable $\xi_n$ from a uniform distribution, 
$$\xi_n(\omega) \sim \text{Unif}\{1,2,..,N\}\,,$$
and update $\tau$ only using the selected sub-loss function:
\begin{equation}\label{eqn: SGD-update}
\tau^{n+1} = \tau^n - \alpha_n \dfrac{\delta L_{\xi_n}}{\delta \tau}[\tau^n]\,.
\end{equation}
Given that $L_i$ only captures the discrepancy associated with the $i$-th data point, this formulation essentially proposes randomly selecting one experiment to conduct for each update.

To deploy Theorem~\ref{thm: SGD converge}, we note that  each $\xi_n$ is independent by construction, and therefore
\begin{equation*}
\begin{aligned}
\mathbb{E}_{\xi_n}\left[\dfrac{\delta L_{\xi_n}}{\delta \tau}[\tau^n]\right] = \sum_{i=1}^N \frac{1}{N} \dfrac{\delta L_{i}}{\delta \tau}[\tau^n] = \dfrac{\delta L}{\delta \tau}[\tau^n]\,,
\end{aligned}
\end{equation*}
which satisfies the unbiased assumption for SGD. The investigation into variance is included in Supplementary Material~SM2. 

The computation of gradient $\frac{\delta L_i}{\delta\tau}$ will be elaborated in Section~\ref{sec: gradient computation}. It will be shown that the computation involves both the forward and adjoint PDE solvers, and thus call for PDE simulation twice.

Step-size is also crucial in updating \eqref{eqn: SGD-update}. Often times, this step size is scheduled to diminish in a specific manner as $n$ increases, gradually reducing variance error. In practical settings, two popular methods are utilized: the Armijo line search condition and AdaGrad, both of which are detailed below.

\vspace{0.1in}
\noindent \textbf{Armijo line search.} 
In each iteration, Armijo line search can be used to adjust the step-size. This is to pre-set $\alpha_n$ at the maximum step-size $\alpha_{\max}$, and to half its value  until the following condition~\eqref{eqn: SGD-armijo rule} is satisfied:
\begin{equation}\label{eqn: SGD-armijo rule}
f_{\xi_n}\left[\tau^n - \alpha_n \frac{\delta f_{\xi_n}}{\delta \tau}[\tau^n]\right] \leq f_{\xi_n}[\tau^n] - c\alpha_n \left\|\frac{\delta f_{\xi_n}}{\delta \tau}[\tau^n]\right\|^2_2\,.
\end{equation}
It is proved in \cite{SGD-armijo} that for convex objective functions, convergence can be achieved,
\[
\mathbb{E}_{\xi}\left[f[\bar{\tau}^n] - f[\tau^*]\right] \leq \frac{C}{n}\left\|\tau^0 - \tau^*\right\|^2\,,
\]
where the expectation $\mathbb{E}_{\xi}[\cdot]$ is taken with respect to all random variables in the first $n$ iterations: $\xi_1,\cdots,\xi_n$.
For general (non-convex) objective functions, under certain properties of objective and with Armijo line-search step size, the algorithm converges in the rate of $O(\frac{1}{n})$:
\[
\min_{1\leq k\leq n} \, \mathbb{E}[\|\nabla f(\tau^k)\|^2]\leq \frac{C}{n}\,.
\]

Deploying Armijo line search strategy in SGD update, we summarize the following Algorithm~\ref{alg: SGD-Armijo}.
\begin{algorithm}
\caption{SGD-Armijo}\label{alg: SGD-Armijo}
\begin{algorithmic}
\REQUIRE{$\tau^0, \varepsilon, \phi, \psi, \alpha_{\max}, c$}
\WHILE{$\|\nabla L[\tau^{n-1}]\|_2 \geq \varepsilon$}
\STATE{Pick $\xi_{n} \sim \text{Unif}(1,N)\,.$}
\STATE{Compute $\frac{\delta L_{\xi_{n}}}{\delta \tau}[\tau^{n-1}]$ by forward and adjoint solver.}
\STATE{Determine step size $\alpha_n$ by Armijo condition \eqref{eqn: SGD-armijo rule}.}
\STATE{Update $\tau^n = \tau^{n-1} - \alpha_n \frac{\delta L_{\xi_{n}}}{\delta \tau}[\tau^{n-1}]\,$.}
\ENDWHILE
\end{algorithmic}
\end{algorithm}

\vspace{0.1in}
\noindent \textbf{Adaptive Gradident (AdaGrad).} 
Another strategy for adjusting the step size is the AdaGrad method~\cite{Duchi-AdaGrad-2011}. It operates as a Gauss-Newton method that utilizes gradient information to approximate the Hessian, and subsequently adjusts the step size automatically based on this Hessian information.
Define the approximate Hessian as:
\begin{equation} \label{eqn: ada-matrix}
 G^n = \sum_{k=1}^n \nabla f_{\xi_k}\otimes \nabla f_{\xi_k}\,,
\end{equation}
the update \eqref{0421} now reads:
\begin{equation}\label{eqn: tau-update-ada}
\tau^{n+1} = \tau^n - \alpha(\delta I + G^n)^{-1/2}\nabla f_{\xi_n}\,,
\end{equation}
with a user-defined small $\delta$ that avoids singularity. The convergence of Adagrad is obtained in 
~\cite{DBBU-Ada-2022}:
\begin{theorem}\label{eqn: adagrad converge}
Under smooth and convex loss function, bounded gradient and certain assumptions on $F$ and $\nabla f$, for total $n$ iterations,
\[\mathbb{E}_\xi[\|\nabla f(k)\|^2] = O\left(\frac{1}{\sqrt{n}}\right)\,,\]
and for non-convex loss function, 
\[\mathbb{E}_\xi[\|\nabla f(k)\|^2] = O\left(\frac{\ln(n)}{\sqrt{n}}\right)\,,\]
i.e. $\liminf \mathbb{E}_\xi[\|\nabla f(k)\|^2 = 0.$
\end{theorem}

Deploying AdaGrad in the framework of SGD, the algorithm in our context is summarized in Algorithm~\ref{alg: SGD-AdaGrad}.
\begin{algorithm}
\caption{SGD-AdaGrad}\label{alg: SGD-AdaGrad}
\begin{algorithmic}[1]
\REQUIRE{$\tau^0, \varepsilon, \alpha, \phi, \psi, \delta$}
\WHILE{$\|\nabla L[\tau^{n-1}]\|_2 \geq \varepsilon$}
\STATE{Pick $\xi_{n} \sim \text{Unif}(1,N)\,.$}
\STATE{Compute $\nabla L_{\xi_{n}}[\tau^{n-1}]$ by forward and adjoint solver.}
\STATE{Update $G_n$ by \eqref{eqn: ada-matrix}.}
\STATE{Update $\tau^n$ by \eqref{eqn: tau-update-ada}.}
\ENDWHILE
\end{algorithmic}
\end{algorithm}

\subsection{Gradient computation}\label{sec: gradient computation}
The computation of the gradient plays a central role in the SGD algorithm for solving our PDE-constrained optimization~\eqref{eqn: PDE-constrained opt}. To do so, we follow the standard calculus of variation strategy. We present the detail of the derivation of the Fr\'echet derivative of loss function  in this subsection. To be more specific, given any source-test pair $(\phi,\psi)$ as described in \eqref{eqn: PDE-constrained opt}, we are to compute the loss function's derivative against the parameter $\tau$.

To do so, we first define the adjoint system:
\begin{equation} \label{adjoint-p2}
\left\{\begin{array}{ll}
\partial_t p + \mu v(\omega) \partial_x p = - \dfrac{1}{\tau(\omega)}\mathcal{L}[p]\,, &\quad 0<x<1 \\[3mm]
p(t=T, x,\mu,\omega) = 0\,, \\[3mm]
p(t, x=0,\mu,\omega) = l \dfrac{h^*}{\mu v\tau\l h^* \r_\omega} \psi(t)\,, &\quad \mu < 0 \\[3mm]
p(t, x=1,-\mu,\omega) = p(t, x=1,\mu,\omega)\,. &\quad \mu<0
\end{array}\right.
\end{equation}
This adjoint equation allows us to explicitly spell out the functional derivative.

\begin{proposition} Given any source-test pair $(\phi,\psi)$ in \eqref{eqn: PDE-constrained opt}, and recall the loss function:
\begin{equation*}
L = \dfrac{1}{2}l^2 = \dfrac{1}{2} \left(\average{\dfrac{\l h \r_{x=0, \mu,\omega}}{\l h^* \r_\omega}\psi}_t - \text{d}\right)^2\,.
\end{equation*}
Its Fr\'echet derivative with respect to $\tau$ is:
\begin{equation}
\begin{split}
\dfrac{\delta L}{\delta \tau}=&\, - \dfrac{l}{\tau^2 \l h^*\r_\omega} \l \phi\,\psi\r_{t,x=0,\mu>0} + \dfrac{l}{\l h^*\r_\omega^2} \dfrac{h^*}{\tau} \l h \psi\r_{t,x=0,\mu,\omega} + \dfrac{1}{\tau h^*} \left\l \mu v p\phi \right\r_{t,x=0,\mu>0}\\
& + \dfrac{1}{\tau h^*}\left\l \mathcal{L}[h] p \right\r_{t,x,\mu} - \dfrac{1}{\tau \average{h^*}_\omega}\left\l \dfrac{h^*}{\l h^*\r_{\omega}} \l h\r_{\mu,\omega} \l p\r_{\mu,\omega} - \average{h}_{\mu,\omega}p\right\r_{t,x,\mu}\,,
\end{split}
\label{eqn: frechet dev}
\end{equation}
where $h$ solves~\eqref{eqn: h-BTE-IBVP} and $p$ solves the adjoint equation~\eqref{adjoint-p2}.
\end{proposition}
\begin{proof}
Assume we give a small perturbation to the relaxation time $\tau$: $\tau \to \tau + \tilde{\tau}$, then $L$ should be accordingly adjusted. By definition of the Fr\'echet derivative, we expect
\[
\delta L = \int\frac{\delta L}{\delta\tau}\tilde{\tau} \,\d\omega\, +\mathcal{O}(\|\tilde{\tau}\|^2_2)\,.
\]
The rest of the proof amounts to explicitly express the perturbation in $L$.

To do so, we notice that both $h^*$ and forward solution are now perturbed:
\begin{equation*}
h^* \to h^* + \tilde{h}^*,\ h \to h+\tilde{h}\,,
\end{equation*}
where we can use the definition of $h^*$ to have $\tilde{h}^* = - h^* \frac{\tilde{\tau}}{\tau}$, and that $\tilde{h}$ solves the leading order approximation:
\begin{equation} \label{eqn: perturb-1}
\left\{\begin{array}{ll}
\partial_t \tilde{h} + \mu v(\omega) \partial_x \tilde{h} = \dfrac{1}{\tau}\mathcal{L}[\tilde{h}] - \dfrac{\tilde \tau}{\tau^2}\mathcal{L}[h]  + \dfrac{1}{\tau}\dfrac{\l h \r_{\mu,\omega}}{\l h^* \r_{\omega}} \left(\tilde{h}^* -\dfrac{\l\tilde{h}^*\r_\omega}{\l h^*\r_\omega}h^*\right)\,, & 0<x<1 \\[3mm]
\tilde{h}(t=0, x,\mu,\omega) = 0\,, \\[3mm]
\tilde{h}(t, x=0,\mu,\omega) = - \phi \dfrac{\tilde{\tau}}{\tau^2}\,, & \mu > 0 \\[3mm]
\tilde{h}(t, x=1,-\mu,\omega) = \tilde{h}(t, x=1,\mu,\omega)\,. & \mu>0
\end{array}\right.
\end{equation}
Multiply the perturbed system \eqref{eqn: perturb-1} with $p$, the adjoint system \eqref{adjoint-p2} with $\tilde h$, add them up and take the average over $(d t,d x,d \mu, \frac{\tau}{h^*}d\omega)$, we have, with integration by parts in space to swap out the boundary conditions:
\begin{equation*}
\begin{aligned}
\dfrac{l}{\l h^*\r_{\omega}} \left\langle \tilde{h}\,,\psi\right\rangle_{t,x=0,\mu<0,\omega} = & \left\langle \tilde{\tau}\,, \dfrac{1}{\tau h^*} \mu v p \phi \right\rangle_{t,x=0,\mu>0,\omega} 
+ \left\l \tilde{\tau}\,, \dfrac{1}{\tau h^*}\mathcal{L}[h]p \right\r_{t,x,\mu,\omega}\\
& 
- \left\l \tilde{\tau}\,, \dfrac{\average{h}_{\mu,\omega}}{\tau \l h^*\r_{\omega}} \left(\dfrac{h^*}{\l h^*\r_{\omega}}\l p\r_{\mu,\omega} - p\right)\right\r_{t,x,\mu,\omega} \,.
\end{aligned}
\end{equation*}
Formally writing $\frac{\delta h}{\delta\tau}= \tilde{h}/{\tilde{\tau}}$, we have:
\begin{equation*}
\begin{aligned}
\dfrac{l}{\l h^*\r_{\omega}} \left\l \dfrac{\delta h}{\delta \tau},\psi \right\r_{t,x=0,\mu<0}  
= & \dfrac{1}{\tau h^*} \left\l \mu v p\phi \right\r_{t,x=0,\mu>0}
+ \dfrac{1}{\tau h^*}\left\l \mathcal{L}[h] p \right\r_{t,x,\mu}\\
& - \dfrac{1}{\tau \average{h^*}_\omega}\left\l \dfrac{h^*}{\l h^*\r_{\omega}} \l h\r_{\mu,\omega} \l p\r_{\mu,\omega} - \average{h}_{\mu,\omega}p\right\r_{t,x,\mu}\,.
\end{aligned}
\end{equation*}
Noticing $L=\frac{1}{2}l^2$, we apply the chain rule for:
\begin{equation}\label{eqn:frechet_preparation}
\begin{aligned}
\dfrac{\delta L}{\delta \tau}(\omega) =& \dfrac{d L}{d \l h \r_{x=0,\mu,\omega}} \dfrac{\delta \l h \r_{x=0,\mu,\omega}}{\delta \tau} + \dfrac{d L}{d \l h^* \r_{\omega}} \dfrac{\delta \l h^* \r_{\omega}}{\delta \tau},\\
= & \dfrac{l}{\l h^* \r_\omega} \left\l \dfrac{\delta \l h\r_{x=0,\mu,\omega}}{\delta \tau},\psi \right\r_{t} -  l\dfrac{\l h\,,\psi\r_{t,x=0,\mu,\omega}}{\l h^*\r_\omega^2} \dfrac{\delta \l h^* \r_{\omega}}{\delta \tau}.
\end{aligned}
\end{equation}
Recalling the definition of $h^*$ and boundary condition for $h$ in~\eqref{eqn: h-BTE-IBVP}, 
\begin{equation*}
\dfrac{\delta \l h \r_{x=0,\mu>0,\omega}}{\delta \tau} = 
- \dfrac{\l\phi\r_{\mu>0}}{\tau^2}\,,\ 
\dfrac{\delta \l h^* \r_{\omega}}{\delta \tau} = - \dfrac{h^*}{\tau}\,.
\end{equation*}
Thus we can write~\eqref{eqn:frechet_preparation} by plugging in the above:
\begin{equation*}
\begin{aligned}
\dfrac{\delta L}{\delta \tau}
= & - \dfrac{l}{\tau^2 \l h^*\r_\omega} \l \phi\,\psi\r_{t,x=0,\mu>0} 
+ \dfrac{l}{\l h^*\r_\omega^2} \dfrac{h^*}{\tau} \l h\, \psi \r_{t,x=0,\mu,\omega} 
+ \dfrac{1}{\tau h^*} \left\l \mu v p\phi \right\r_{t,x=0,\mu>0}\\
& + \dfrac{1}{\tau h^*}\left\l \mathcal{L}[h] p \right\r_{t,x,\mu} - \dfrac{1}{\tau \average{h^*}_\omega}\left\l \dfrac{h^*}{\l h^*\r_{\omega}} \l h\r_{\mu,\omega} \l p\r_{\mu,\omega} - \l h\r_{\mu,\omega} p\right\r_{t,x,\mu}\,.
\end{aligned}
\end{equation*}
\end{proof}    
\begin{remark}
We should note that the collision operator for $h$ is self-adjoint, so its associated adjoint-equation is relatively simple. It is also possible to compute the Fr\'echet derivative using $g$ in~\eqref{eqn: g-BTE-IBVP} and its associated adjoint equation, but since $\mathcal{L}_0$ is not a self-adjoint operator, the computation can be much more convoluted. For the conciseness of the computation, we stick with the $(h,p)$ pair.
\end{remark}

The derivation of adjoint system follows the standard Lagrange multiplier technique. Let $p(t,x,\mu,\omega)$ be the Lagarange multiplier, $p_1(x,\mu,\omega)$ the multiplier corresponding to initial condition, $p_2(t,\mu^+,\omega), p_3(t,\mu^+,\omega)$ the multipliers corresponding to boundary conditions.
\begin{equation*}
\begin{aligned}
\mathcal{L}[h,p,p_i] :=\ & \dfrac{1}{2}(\Delta l)^2 + \left\l \partial_t h + \mu v \partial_x h - \dfrac{1}{\tau}\left(\frac{\l h \r_{\mu,\omega}}{\l h^* \r_{\omega}}h^* - h\right),\ p \right\r_{t, x, \mu, \omega} \\
& + \l h, p_1\r_{0,x,\mu,\omega}
+ \left\l h - \frac{\phi}{\tau}, \ p_2\right\r_{t, 0, \mu>0, \omega} 
+ \l h(-\mu) - h(\mu),\ p_3 \r_{t,1,\mu>0,\omega}.
\end{aligned}
\end{equation*}
With this new Lagrangian, the original optimization problem~\eqref{eqn: PDE-constrained opt} is reformulated as
\[
\min_{h,p,p_i} \mathcal{L}[h,p,p_i]\,.
\]
with its argument living on the extended space. The first-order condition of the optimizer requires $\frac{\delta\mathcal{L}}{\delta h}=0$ pointwise in $(t,x,\mu,\omega)$. From the above system we obtain
\begin{align} 
\left\{\begin{array}{ll}
- \partial_t p - \mu v(\omega) \partial_x p  - \left(\dfrac{1}{\l h^* \r_{\omega}} \l \dfrac{h^*p}{\tau}\r_{\mu,\omega} - \dfrac{p}{\tau}\right) = 0, \\[3mm]
p(t=T,x,\mu,\omega) = 0, \\[3mm]
p_3(t,x=1,\mu,\omega) = \mu v(\omega) p(t,1,\mu,\omega), & \mu>0 \\[3mm]
p_3(t, x=1,-\mu,\omega) = -\mu v(\omega) p(t,1,\mu,\omega), & \mu<0.\\[3mm]
p (t,x=0,\mu, \omega) = \dfrac{l}{\mu v(\omega)\l h^* \r_\omega} \psi(t), & \, \mu<0 \\[3mm]
p(t=0,x,\mu,\omega) = p_1(t=0,x,\mu,\omega), \\[3mm]
p_2(t,x=0,\mu,\omega) = \mu v(\omega) p(t,x=0,\mu,\omega) - l\dfrac{1}{\l h^*\r_\omega} \psi(t), & \, \mu>0.
\end{array}\right.
\end{align}
Organize the equations above, we arrive at the adjoint system spelled out in~\eqref{adjoint-p2}.

\section{Numerical results} \label{sec: numerics}
We present several numerical examples to illustrate the reconstruction process. The numerical setup considers $t\in [0,1.5]$, $x\in [0,1]$, $\mu\in(-1,1)$, with the frequency range truncated to $\omega\in [0.4,4]$, assuming that higher frequencies contribute negligibly. The discretization parameters are chosen as $(\Delta x,\Delta t,\Delta\omega)=(0.02, 0.005, 0.4)$, which provide sufficient numerical accuracy. For discretization in $\mu$, we employ Gauss-Legendre quadrature with $N_\mu = 64$. The time-stepping is implemented using the simple forward Euler method, and the spatial discretization employs an upwind scheme. The group velocity is given by $v(\omega) = 2.5 - 0.2\omega$. In this discrete setting, $\tau(\omega)$ becomes a $10$-dimensional vector to be reconstructed.

\subsection{Numerical study of the forward equation} \label{sec: numerical forward eqn}
We first set $\tau$ at the ground-truth $\tau^*$:
\[
\tau^*(\omega) = \dfrac{1}{\sqrt{5\omega}} + 1\,,
\]
and we run the forward equation~\eqref{eqn: h-BTE-IBVP} (or equivalently, equation~\eqref{eqn: g-diff-regime} with $\varepsilon=1$) to study the solution profile for $h$. Figure~\ref{fig: truth and equilibrium} (a) shows the profile of the ground-truth $\tau^*$ and Figure~\ref{fig: truth and equilibrium} (b) presents the equilibrium distribution.
\begin{figure}[htbp]
    \centering
    \begin{subfigure}[b]{0.35\textwidth}
         \centering
         \includegraphics[width=\textwidth]{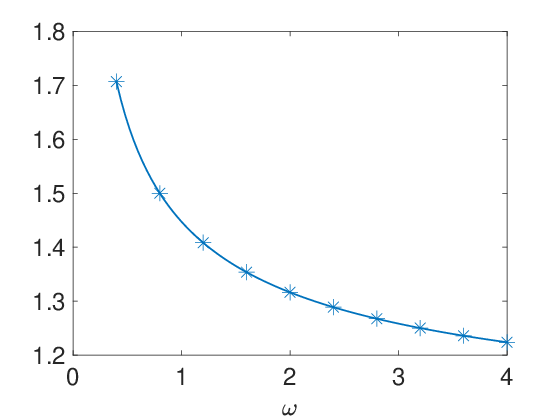}
         \caption{Groundtruth: $\tau^*(\omega)\,.$}
         \label{fig: truth tau}
    \end{subfigure}
    \hspace{1.2cm}
    \begin{subfigure}[b]{0.35\textwidth}
         \centering
         \includegraphics[width=\textwidth]{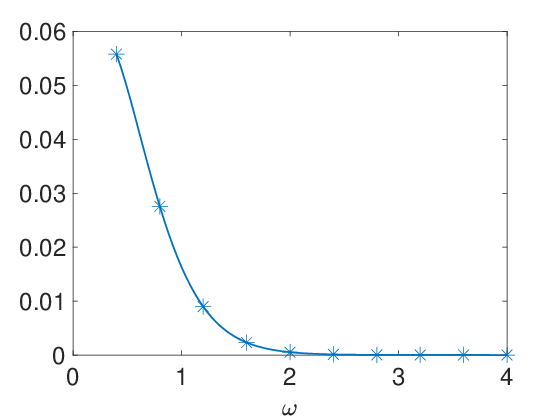}
         \caption{Equilibrium: $h^\ast(\omega)\,.$}
         \label{fig: h star eq}
    \end{subfigure}
    \caption{$\tau^\ast$ and $h^\ast\,.$}
    \label{fig: truth and equilibrium}
\end{figure}

Different panels in Figure~\ref{fig: forward solution} demonstrate the solution $\average{h}_\mu$ at various time slides. Initially the system is set at rest and a narrow Gaussian source is injected around $t_0=0.04$ and concentrates at $\mu_0 = 0.96$, $\omega_0 = 2$. From panel (a)-(c), it is seen that the solution is predominantly governed by its ballistic component, with the solution profile moving towards the right hand side boundary $x=1$ before being bounced back, as shown in panel (d). After the solution hits back to $x=0$, it dissipates out from the left boundary.

\begin{figure}[htbp]
     \centering
     \begin{subfigure}[b]{0.32\textwidth}
         \centering
         \includegraphics[width=\textwidth]{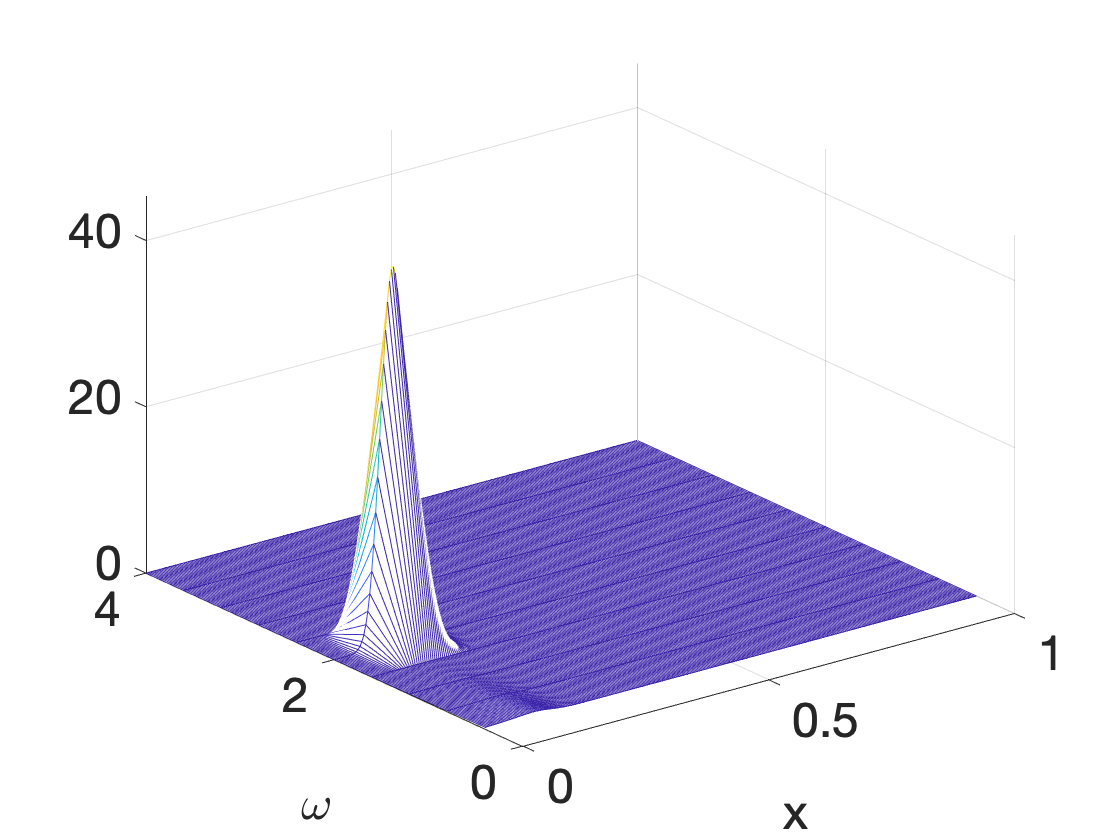}
         \caption{$\langle h \rangle_{\mu}(t=0.1)$}
     \end{subfigure}
     \hfill
     \begin{subfigure}[b]{0.32\textwidth}
         \centering
         \includegraphics[width=\textwidth]{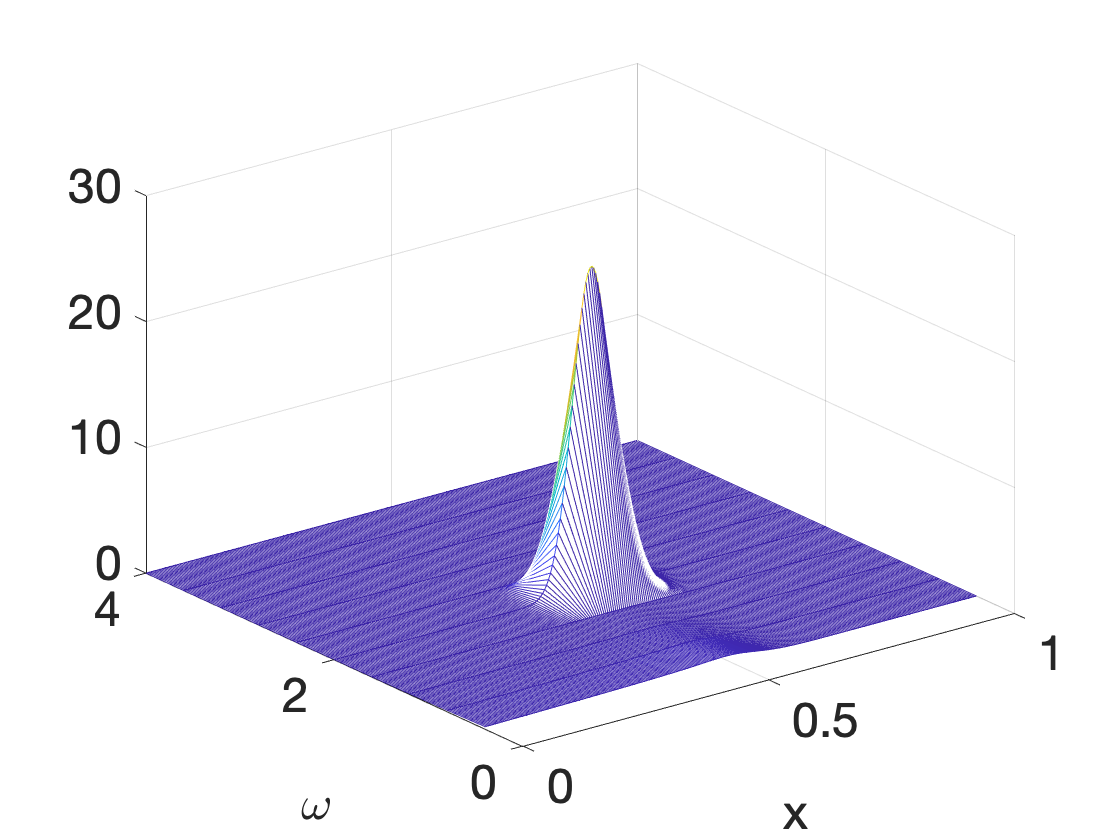}
         \caption{$\langle h \rangle_{\mu}(t=0.3)$}
     \end{subfigure}
     \hfill
     \begin{subfigure}[b]{0.32\textwidth}
         \centering
         \includegraphics[width=\textwidth]{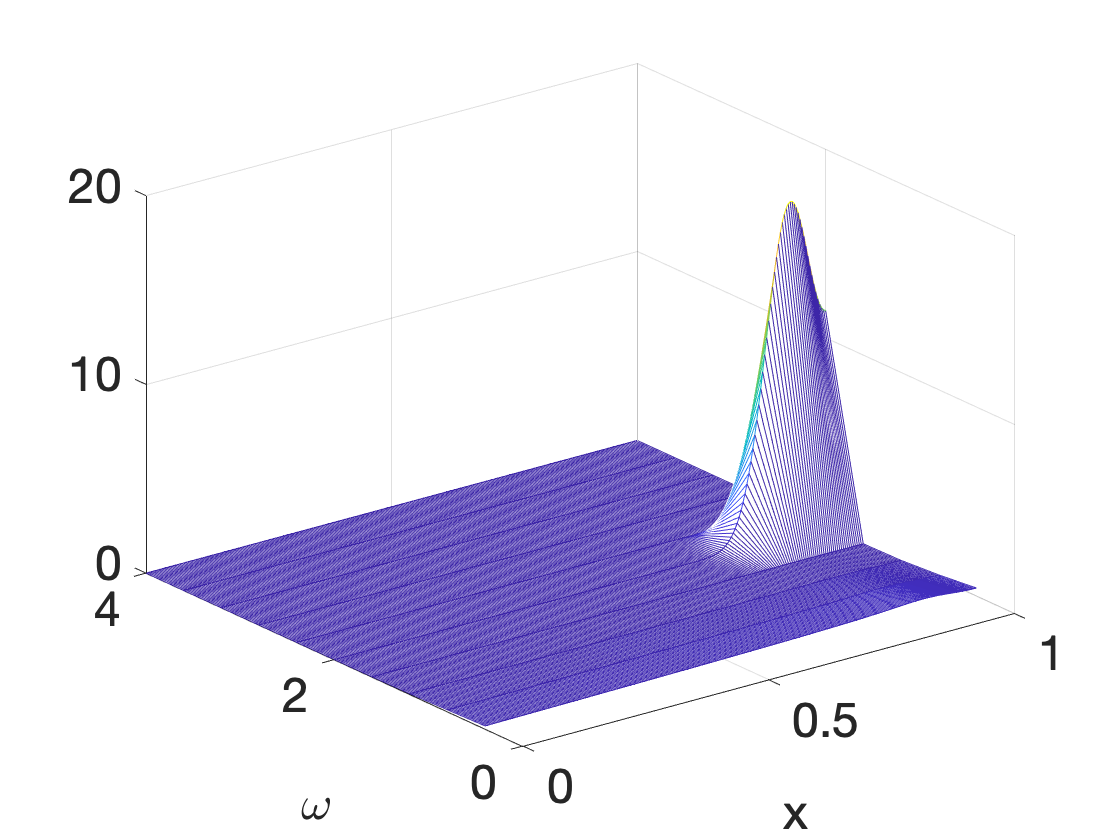}
         \caption{$\average{h}_{\mu}(t=0.5)$}
     \end{subfigure}
     \hfill
     \\
     \begin{subfigure}[b]{0.32\textwidth}
         \centering
         \includegraphics[width=\textwidth]{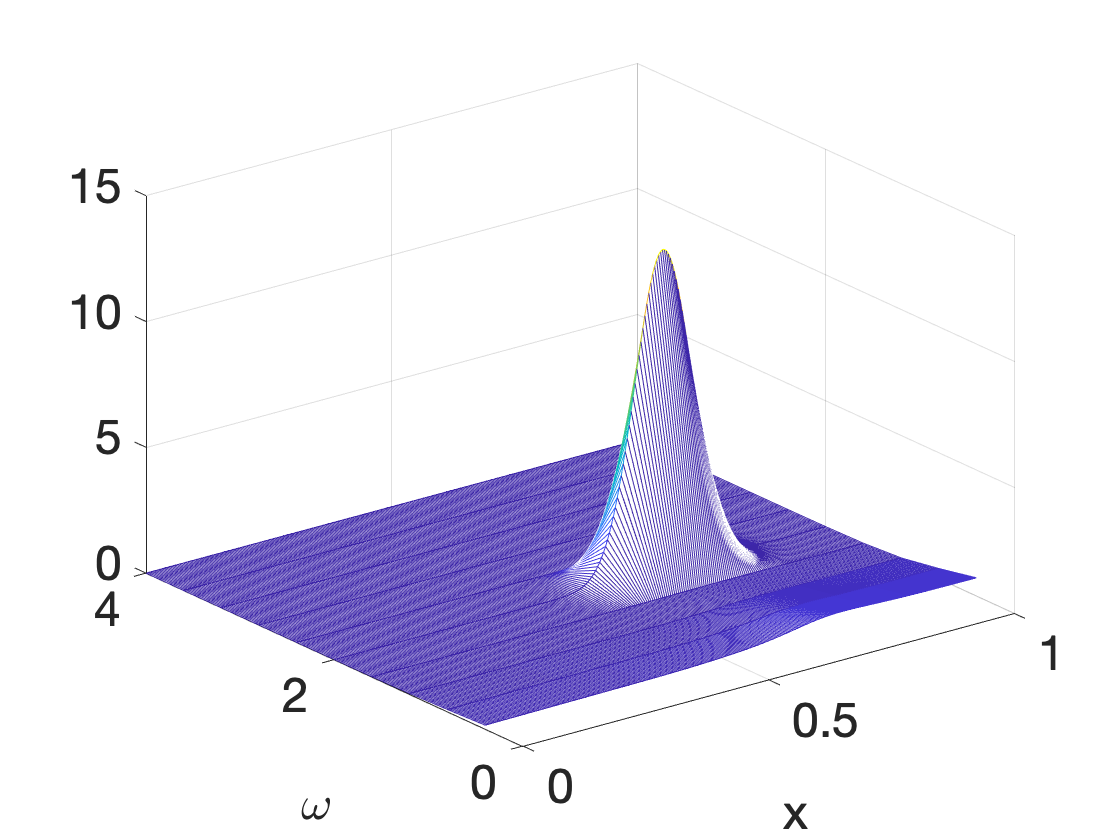}
         \caption{$\langle h \rangle_{\mu}(t=0.7)$}
     \end{subfigure}
     \hfill
     \begin{subfigure}[b]{0.32\textwidth}
         \centering
         \includegraphics[width=\textwidth]{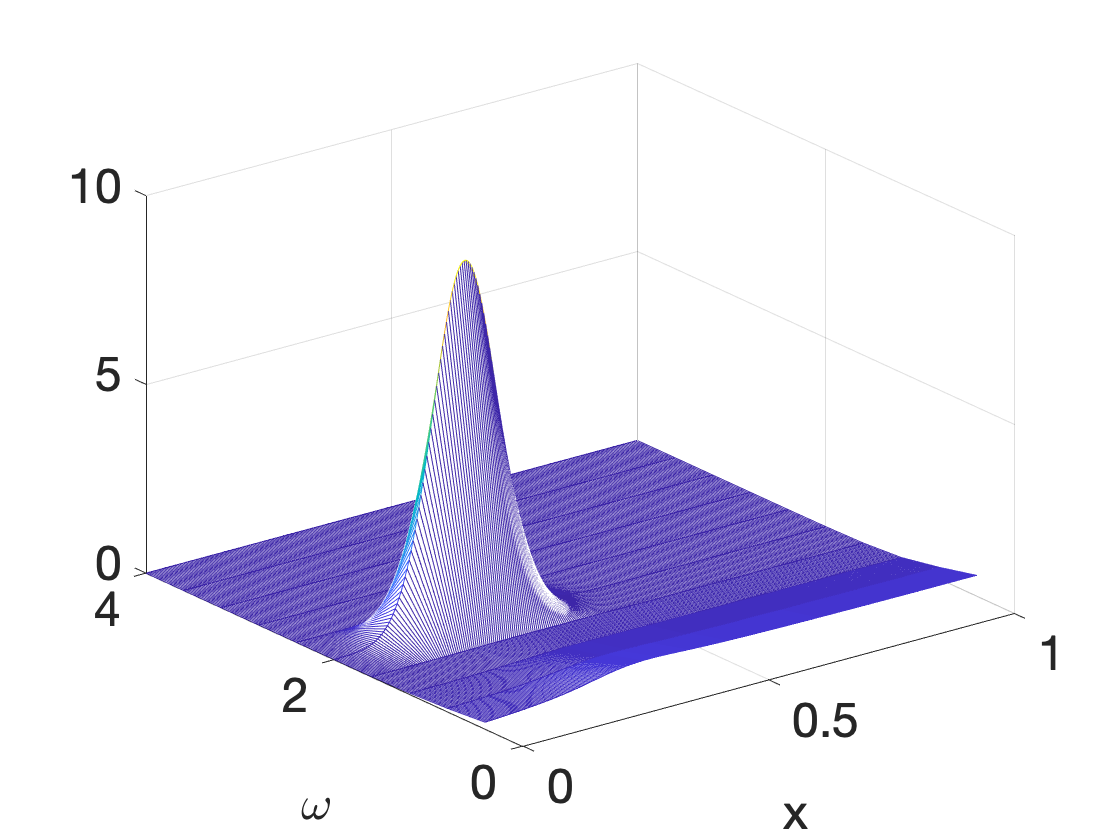}
         \caption{$\langle h \rangle_{\mu}(t=0.9)$}
     \end{subfigure}
     \hfill
     \begin{subfigure}[b]{0.32\textwidth}
         \centering
         \includegraphics[width=\textwidth]{ 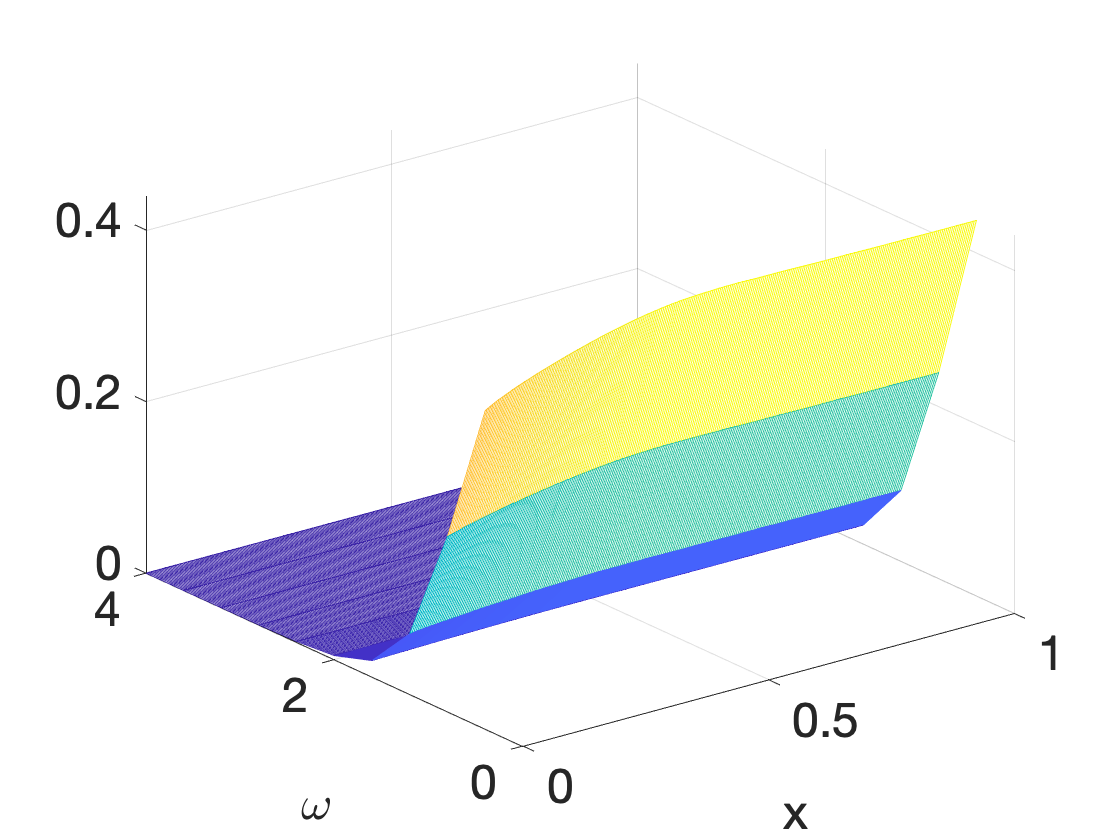}
         \caption{$\langle h \rangle_{\mu}(t=1.2)$}
     \end{subfigure}
     \caption{Forward solution snapshots for $\varepsilon = 1\,.$}
     \label{fig: forward solution}
\end{figure}

Since the ballistic feature dominates the solution, it is straightforward to predict the arrival time of the (bulk of) bounced-back solution. In particular, a source function that is a smooth Gaussian concentrating in $t_0, \mu_0, \omega_0$, its bulk is expected to arrive back at $x=0$ at the following time:
\begin{equation}\label{eqn:arrival_time}
t_R(t_0,\mu_0,\omega_0) := t_0 + \frac{2}{\mu_0 v(\omega_0)}\,. 
\end{equation}
In this specific example, we plug in the values of the choice of $t_0,\mu_0$ and $\omega_0$ to obtain $t_R=1.0321$. In Figure \ref{fig: T measurement at x=0}, we plot the temperature measurement at $x=0$. It is clear there is a temperature bump exactly at $t_R$, verifying that the bulk of the bounced-back solution has arrived at $x=0$ boundary.
\begin{figure}[htbp]
     \centering
     \begin{subfigure}[b]{0.32\textwidth}
         \centering
         \includegraphics[width=\textwidth]{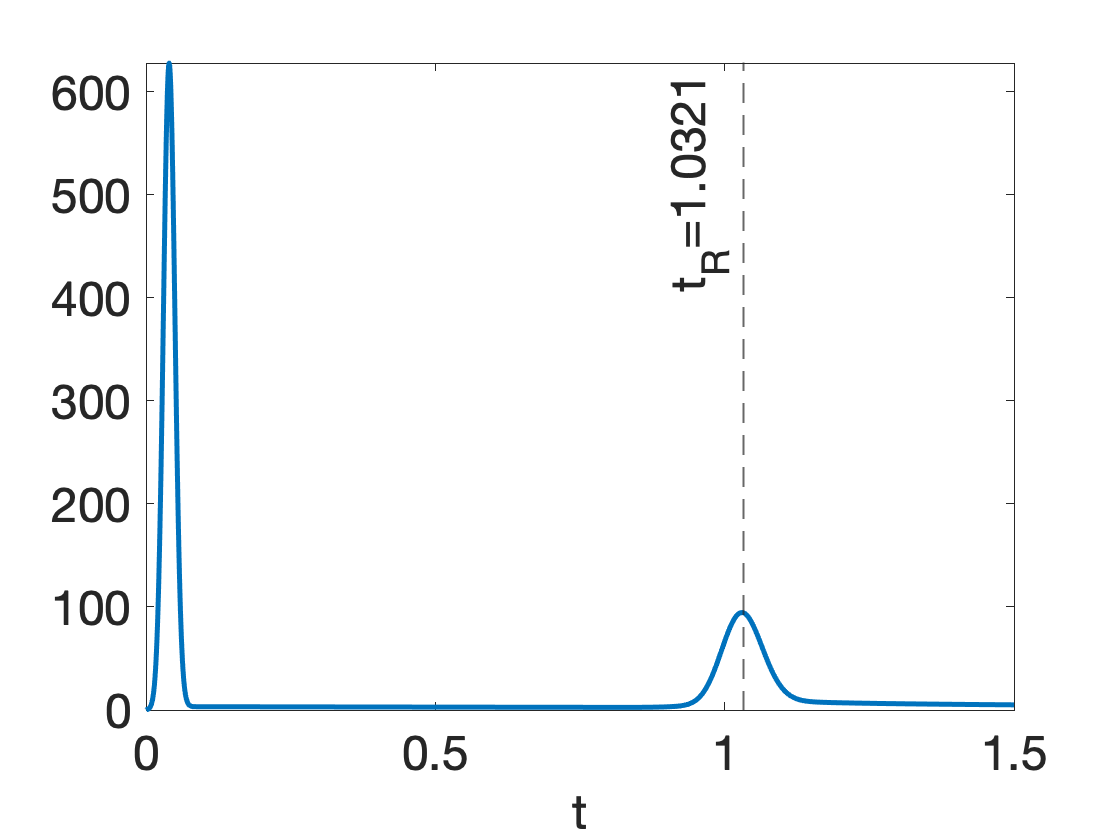}
     \end{subfigure}
     \caption{Temperature measurement at the left boundary: $T(t,x=0)\,.$}
     \label{fig: T measurement at x=0}
\end{figure}

We next test the solution's dependence on $\eps$. We re-run the dimensionless phonon transport equation~\eqref{eqn: g-diff-regime} with $\varepsilon=0.1$.
The solution at various time slides are shown in Figure~\ref{fig: forward solution, eps = 0.1}. The same source, initial and reflective boundary conditions are deployed. Panel (a) still presents some concentration of the solution profile, but the equation quickly dissipates and demonstrates equilibrium behavior comparing with the case when $\varepsilon=1$ in Figure~\ref{fig: forward solution}. In particular, panel (d) shows a slice of $h(\omega)$ at $(t,x,\mu)=(0.16,0.5,0.9675)$. It is obvious that the solution in $\omega$ resembles the local equilibrium (up to a constant), when compared with Figure~\ref{fig: truth and equilibrium} (b).

\begin{figure}[htbp]
     \centering
     \begin{subfigure}[b]{0.35\textwidth}
         \centering
         \includegraphics[width=\textwidth]{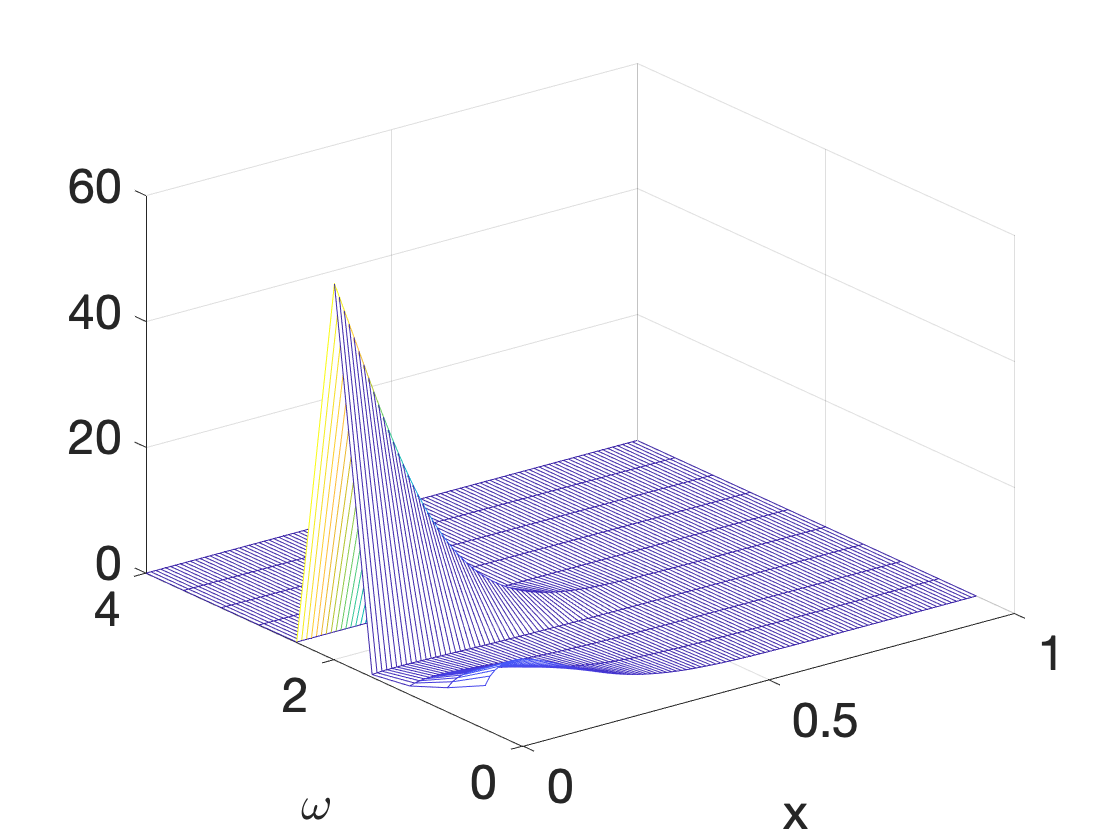}
         \caption{$\average{h}_{\mu}(t=0.04)$}
     \end{subfigure}
     \hspace{1.2cm}
     \begin{subfigure}[b]{0.35\textwidth}
         \centering
         \includegraphics[width=\textwidth]{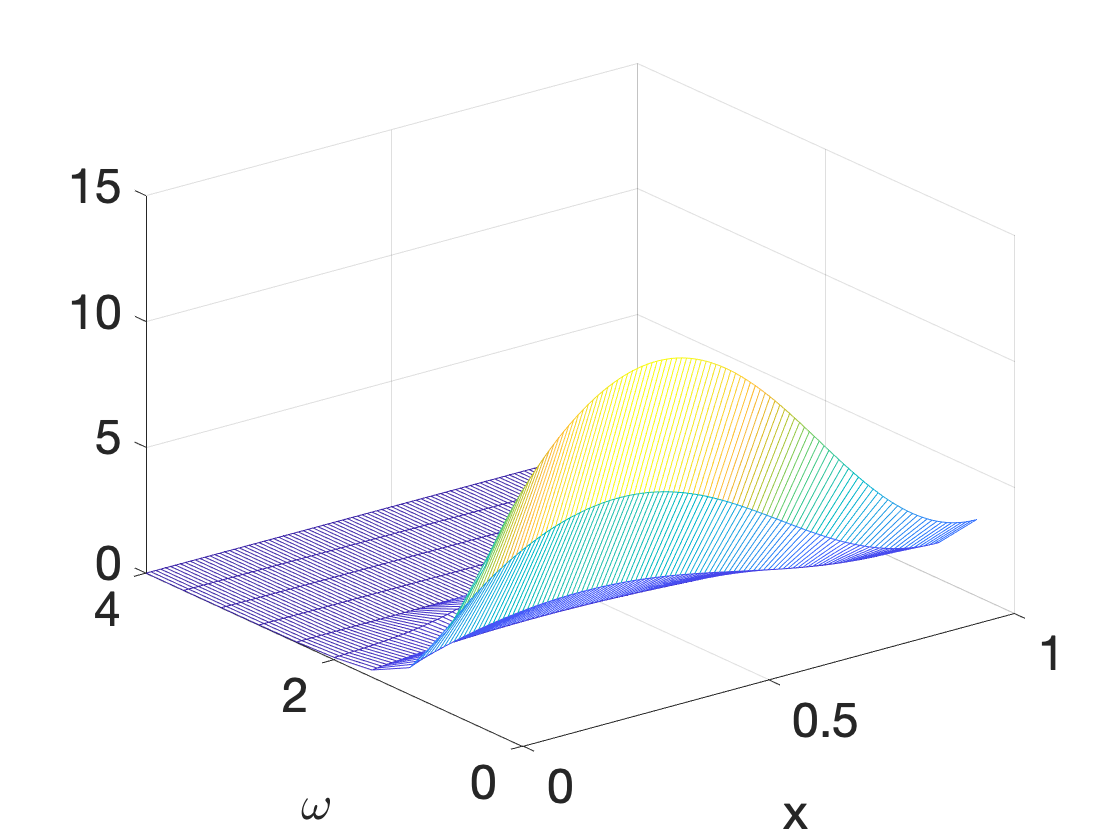}
         \caption{$\langle h \rangle_{\mu}(t=0.08)$}
     \end{subfigure}
     \\
     \begin{subfigure}[b]{0.35\textwidth}
         \centering
         \includegraphics[width=\textwidth]{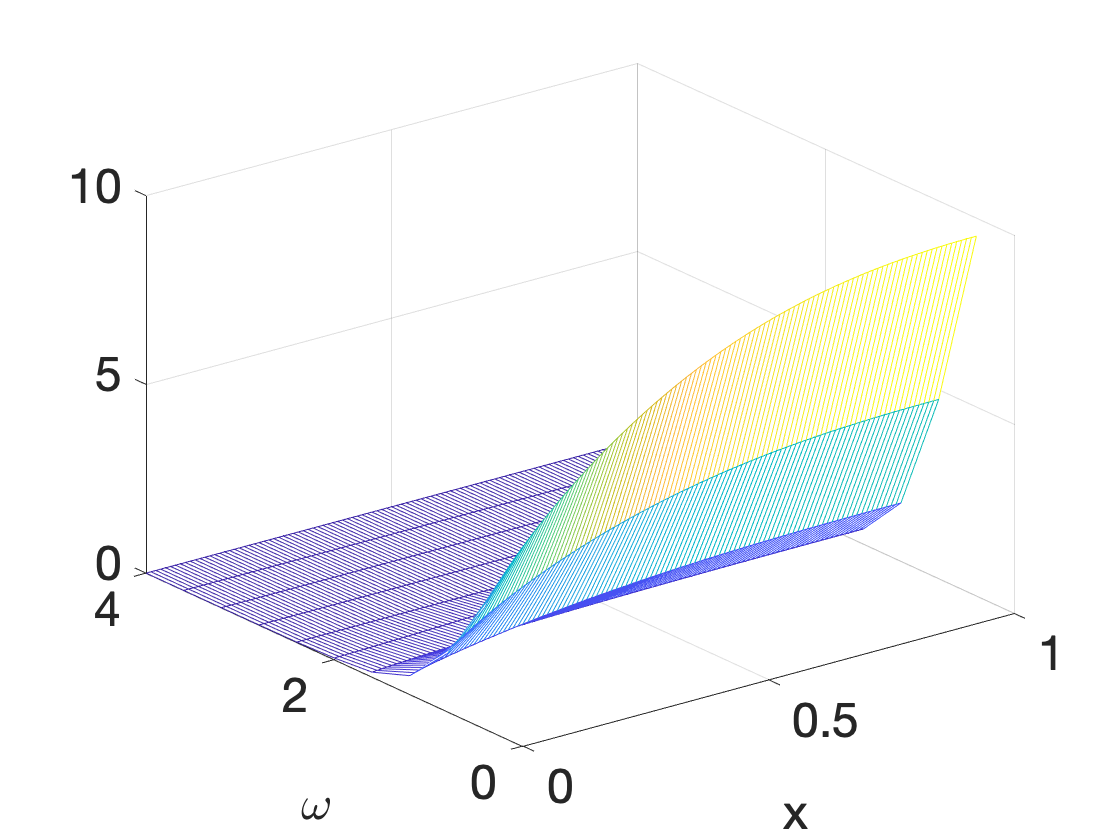}
         \caption{$\langle h \rangle_{\mu}(t=0.12)$}
     \end{subfigure}
     \hspace{1.2cm}
     \begin{subfigure}[b]{0.35\textwidth}
         \centering
         \includegraphics[width=\textwidth]{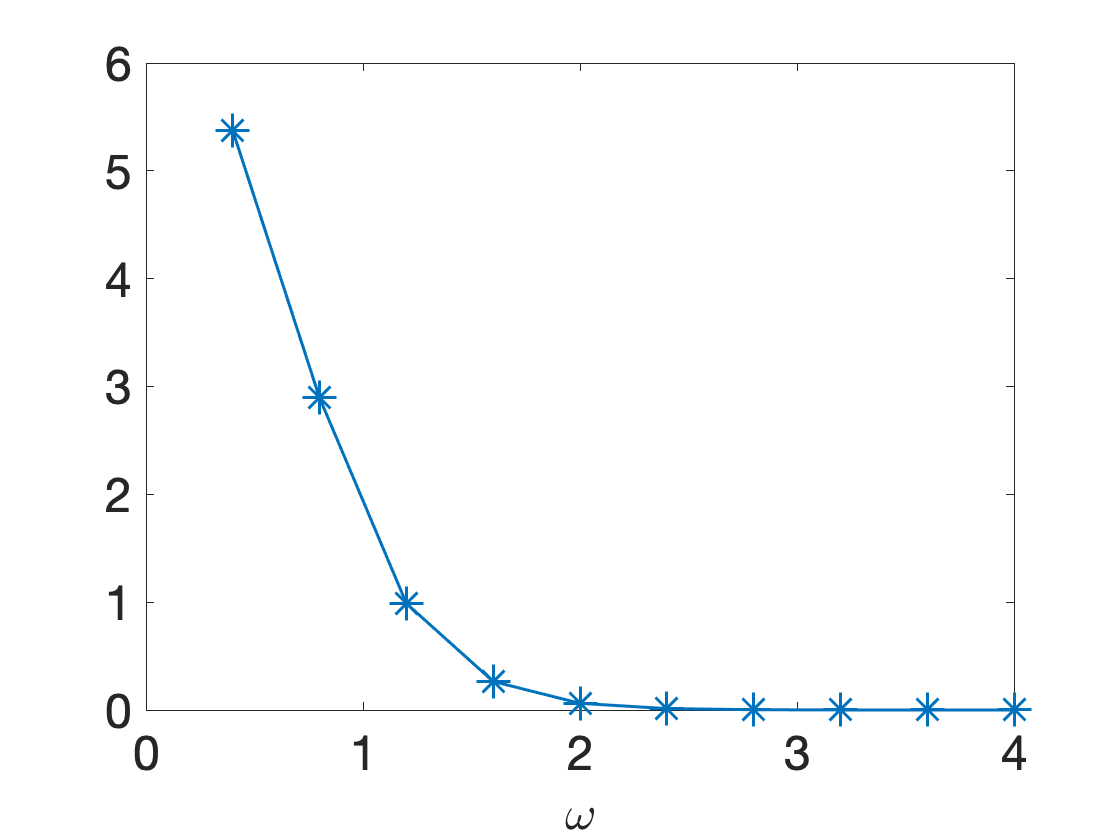}
         \caption{$h(0.12, 0.5, 0.9675, \omega)$}
     \end{subfigure}
     \hfill
     \caption{Forward solution snapshots for $\varepsilon = 0.1\,.$}
     \label{fig: forward solution, eps = 0.1}
\end{figure}

\subsection{Numerical study of the inverse problem}
We now switch gear towards the inverse problem as described in section~\ref{subsec: inverse setup}, and reconstruct the relaxation coefficient $\tau$ using measured data in the ballistic regime ($\varepsilon=1$). To do so, we deploy the optimization methods discussed in Section~\ref{sec: SGD algs}. Both SGD-Armijo and SGD-AdaGrad are implemented.

The data is generated by calling~\eqref{eqn:real_data}. In particular, we send signal of the form:
\[
\phi_i = \phi_{\eps_1}(t-t_0)\phi_{\eps_2}(\mu-\mu_0) \phi_{\eps_3}(\omega-\omega_i)\,,\quad i = 1\,,\cdots, 10\,,
\]
to be the boundary condition. Numerically $(\eps_1,\eps_2,\eps_3)$ are set to be $(0.01, 0.01, 0.1)$. $(t_0,\mu_0)$ is set to be $(0.1, 0.93)$, and $\omega_i$ takes on all values on the grids $0.4+i\Delta\omega$. For each $\phi_i$, the associated temperature measurement is taken as
\[
\mathrm{d}_i=\Lambda_{\tau^\ast}(\phi_i,\psi_i) = \average{\frac{\average{h_i}_{x=0,\mu,\omega}}{\average{h^*}_\omega} \,,\psi_i}_t\,,
\]
where $h_i$ is the PDE solution ran by the groundtruth $\tau^*$, and $\psi_i$ is the testing window function
\[
\psi_i = \phi_{\eps_4}(t-t_R(t_0, \mu_0, \omega_i))\,,
\]
where $t_R$ is defined in~\eqref{eqn:arrival_time} and $\eps_4$ is set to be $0.08$. In this numerical setting, we have in total $10$ data pairs to reconstruct a $10$-dimensional object $\sigma$.

As an initial guess, we set $\tau$ to be:
\begin{equation*}
\tau^0(\omega) = -0.15 (\omega - 4) + 1.4\,.
\end{equation*}
It is a linear decreasing function that is reasonably different from the groundtruth $\tau^\ast$. To measure the convergence, we quantify both the mismatch function $L[\tau^n]$ defined in~\eqref{eqn: PDE-constrained opt}, and the error of the reconstruction: the $L^2$ distance between the current profile ($\tau^n$) and the groundtruth $\tau^*$:
\begin{equation*}
e^n := \|\tau^n - \tau^\ast\|_2 = \frac{1}{\sqrt{N_\omega}}\left(\sum_{i=1}^{N_\omega}(\tau^n(\omega_i) - \tau^*(\omega_i))^2\right)^{1/2}\,.
\end{equation*}

In Figure~\ref{fig: SGD-profile update}, we plot the reconstruction by SGD-Armijo and SGD-AdaGrad in the two panels. For each algorithm, we show the reconstruction at $8$ different time slides. Both algorithms roughly converge at $n=500$. It is worth noting that the algorithms are stochastic in nature and different coordinates get updated with different rates.

\begin{figure}[htbp]
    \centering
    \begin{subfigure}[b]{0.45\textwidth}
         \centering
         \includegraphics[width=\textwidth]{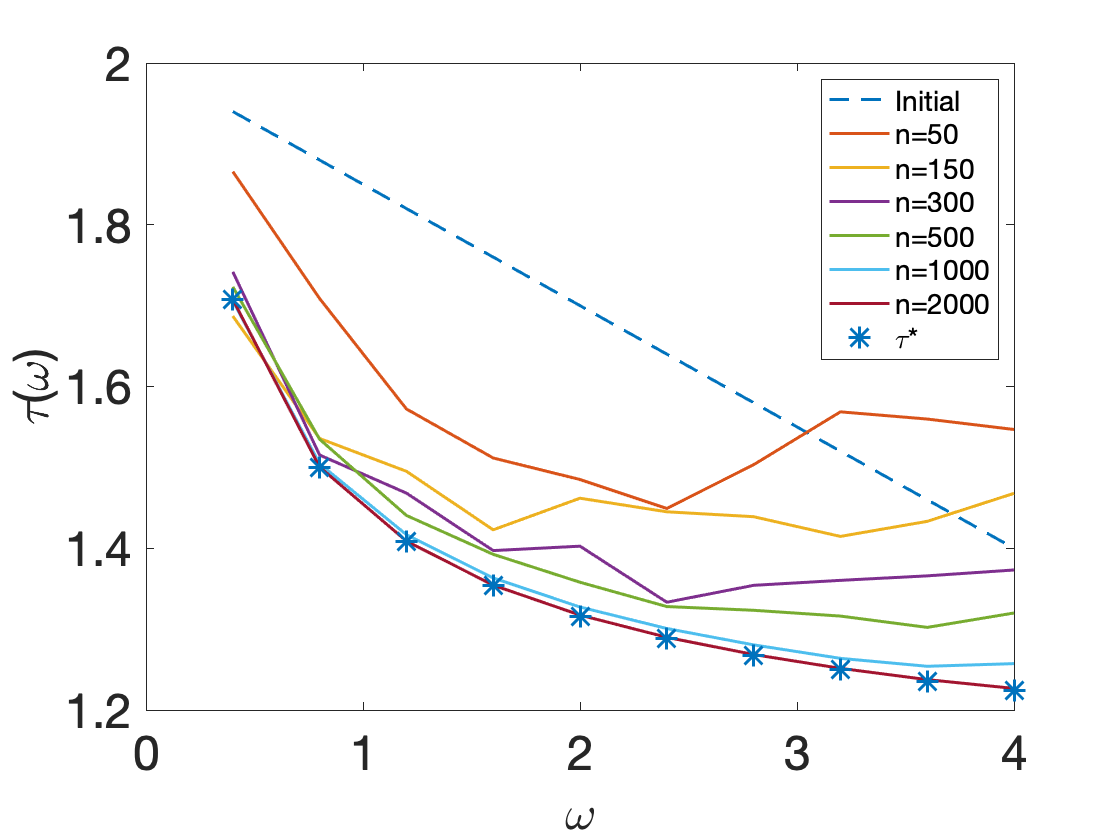}
         \caption{SGD-Armijo}
         \label{fig: SGD-profile1}
     \end{subfigure}
     \hfill
     \begin{subfigure}[b]{0.45\textwidth}
         \centering
         \includegraphics[width=\textwidth]{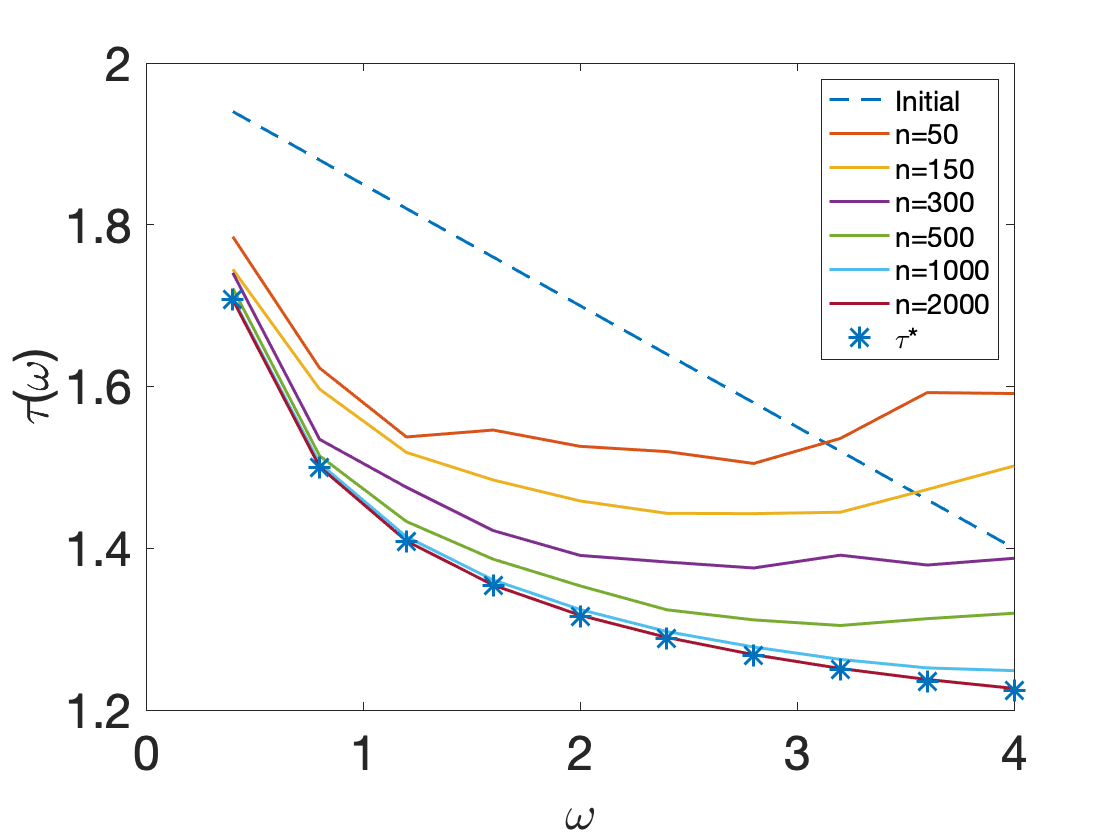}
         \caption{SGD-Adagrad}
         \label{fig: SGD-profile2}
     \end{subfigure}
     \caption{$\tau^n$ profile update.}
     \label{fig: SGD-profile update}
\end{figure}

In Figure~\ref{fig: errors}, we study the convergence of the algorithms. The mismatch and the error are plotted in Panel A and B respectively, with each plot demonstrating the convergence generated by two different algorithms. The plots are presented in semi-log scale, and they are roughly linear, suggesting an exponential decay.

\begin{figure}[htbp]
     \centering
     \begin{subfigure}[b]{0.45\textwidth}
         \centering
         \includegraphics[width=\textwidth]{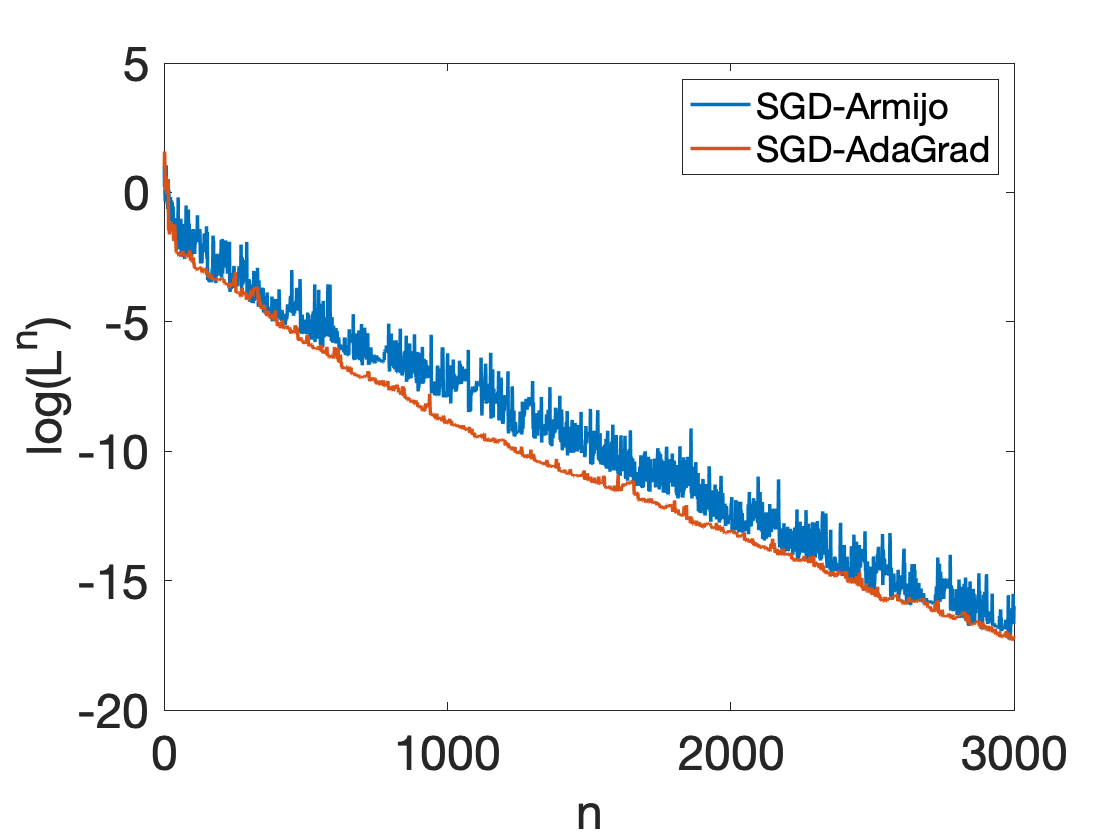}
         \caption{Mismatch $L[\tau^n]$ in semi-log scale.}
         \label{fig: SGD-profile2}
     \end{subfigure}
     \hfill
     \begin{subfigure}[b]{0.45\textwidth}
         \centering
         \includegraphics[width=\textwidth]{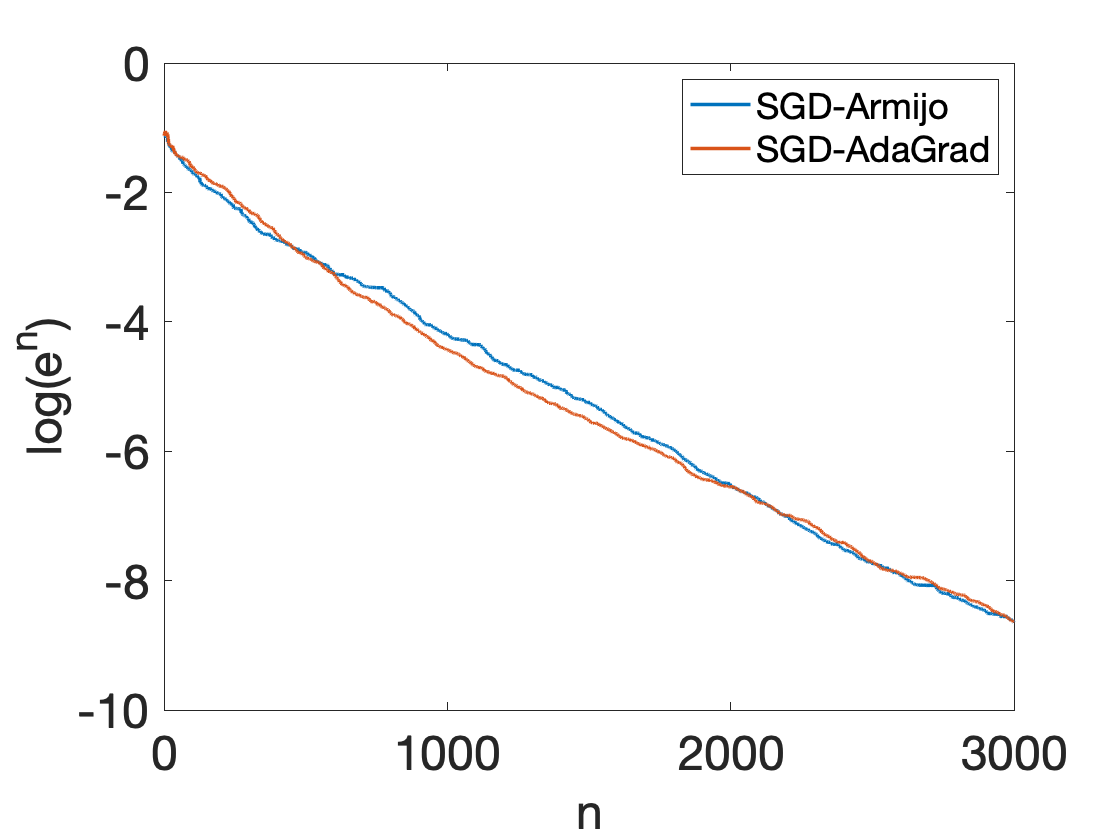}
         \caption{Error $e^n$ in semi-log scale.}
         \label{fig: SGD-error2}
     \end{subfigure}
     \caption{Mismatch and error.}
     \label{fig: errors}
\end{figure}

\subsection{Numerical study of the loss function} \label{sec: numerical loss fcn}
In this subsection, we note that, although not the focus of the current paper, different choices of $\phi_i$and $\psi_i$ can lead to varying sensitivities of the loss function with respect to $\tau$. In our set up, the source functions that we chose are very concentrated beams, and the test window functions are positioned at times that capture the temperature response when the majority of the beam is reflected. As a result, the loss function is expected to be sensitive to $\tau$.

To be more specific, recalling the profile of $\phi_i$ defined in~\eqref{eqn:phi_cond}, we see that it is concentrated at $\omega_i$. As shown in  Figure~\ref{fig: forward solution}, the profile of the solution remain concentrated around $\omega_i$ throughout the evolution. Therefore, the contribution of the loss function from this specific data point $L_i = \frac{1}{2} |\Lambda_\tau(\phi_i, \psi_i) - \Lambda_{\tau^*}(\phi_i,\psi_i)|^2$ is sensitive to changes in $\tau(\omega_i)$. To numerically verify this argument, we compute the gradient of $L_i$ with respect to $\tau$ for various of $i$. In Figure~\ref{fig: show gradients}, we present these gradients evaluated at $\tau^0$. It is obvious that for every $i$, the gradient of $\nabla_\tau L_i$ takes a peak at $\omega_i$, suggesting $L_i$ is sensitive to $\tau(\omega_i)$.

Apart from sensitivity, properties such as Lipschitz continuity and comparison between $\nabla L_i, \nabla L_j$ are discussed in Supplementary Material~SM2. 
\begin{figure}[htbp]
    \centering
    \includegraphics[scale =0.2]{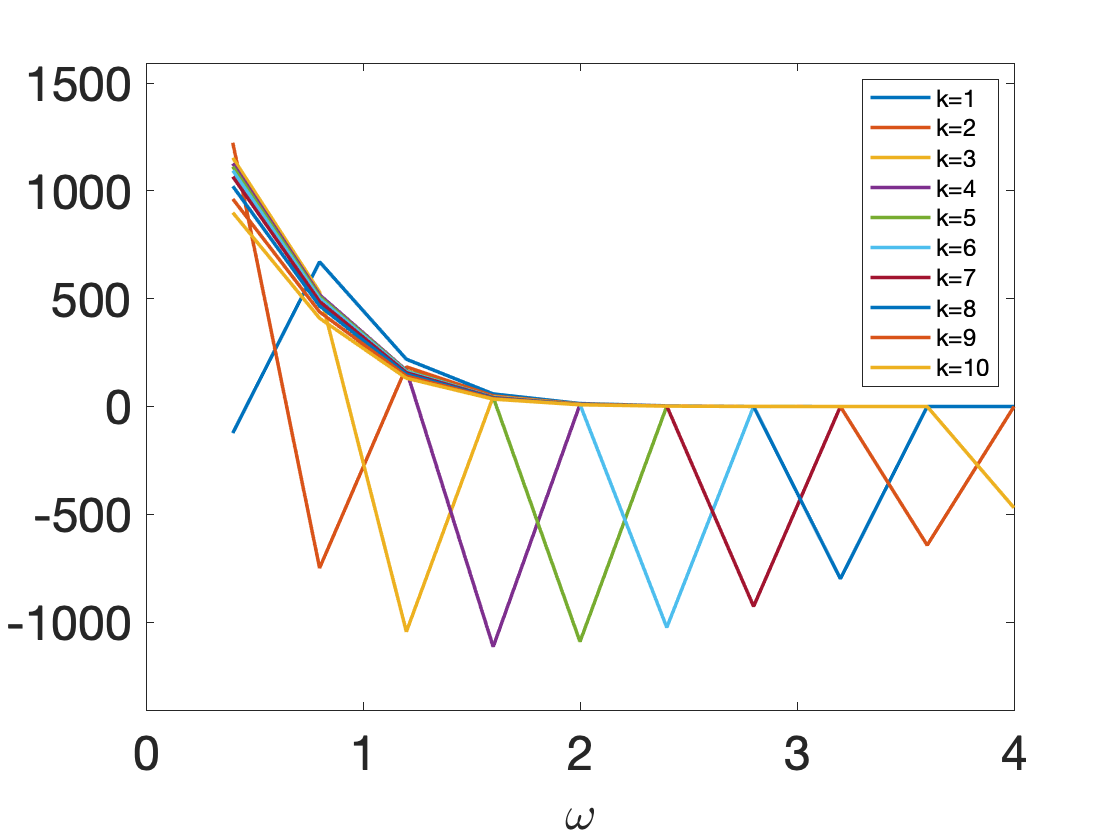}
    \caption{$\nabla L_k[\tau^0]\,.$}
    \label{fig: show gradients}
\end{figure}

\appendix

\section{Physical quantities}
We carefully check the units and cite the values of some physical quantities defined throughout the paper, listed in Table~\ref{tbl: units}. It can be verifed that they are consistent with physical definition~\cite{Majumdar-1993-heat conduct}.
\begin{table}[h]
\footnotesize
\centering
\caption{Units of physical quantities.}
\begin{tabular}{|c|c|c|} \hline
Notation & Physical meaning & Value/units\\ \hline
$\hbar$ & Reduced Planck constant & $1.055\times 10^{-34} \mathrm{m}^2\!\cdot\!\mathrm{kg}\!\cdot\! \mathrm{s}^{-1}$ \\
$k_B$ & Boltzmann constant & $1.381\times 10^{-23} \mathrm{J}\!\cdot\!\mathrm{K}^{-1}$ \\
$T_0$ & reference temperature & $300 \mathrm{K}$\\
$f$ & phonon statistical distribution & scalar\\
$g$ & phonon energy distribution & $\mathrm{J} \!\cdot\! \mathrm{s}\!\cdot\! \mathrm{m}^{-3}$ \\
$g^*(\omega)$ & mode specific heat & $\mathrm{J}\!\cdot\! \mathrm{s} \!\cdot\! \mathrm{m}^{-3}\!\cdot\! \mathrm{K}^{-1}$ \\
$\mu$ & direction & scalar\\
$\omega$ & phonon frequency & $(0,10)\,\mathrm{THz} = (10^{12}, 10^{13})\, \mathrm{s}^{-1}$ \\
$v(\omega)$ & group velocity & $\mathrm{m}\!\cdot\!\mathrm{s}^{-1}$ \\
$\tau(\omega)$ & relaxation time & $\mathrm{ns}$\\
$\Lambda$ & Mean free path & $\mu\mathrm{m}$\\
$D(\omega)$ & density of states & $\mathrm{m}^{-3}\!\cdot\!\mathrm{s}$\\
$\average{g^*}_\omega$ & volumetric heat capacity & $\mathrm{J}\!\cdot\! \mathrm{m}^{-3}\!\cdot\! \mathrm{K}^{-1}$\\ 
$\frac{\average{g/\tau}_{\mu,\omega}}{\average{g^*/\tau}_{\mu,\omega}} (T)$ & temperature & $K$ \\
$\average{\mu v g}_{\mu,\omega} (q)$ & heat flux & $\mathrm{J}\!\cdot\! \mathrm{m}^{-2} \!\cdot\!\mathrm{s}^{-1}$\\
$\kappa$ & heat conductivity & $\mathrm{J}\!\cdot\! \mathrm{s}^{-1} \!\cdot\!\mathrm{m}^{-3}\!\cdot\!\mathrm{K}^{-1}$ \\ \hline
\end{tabular}
\label{tbl: units}
\end{table}
As indicated in Table~\ref{tbl: units},  $\frac{\hbar \omega}{k_B T} = O(1)$, we remark that Bose-Einstein distribution is in the order of: $f_\BE=O(1)$. Hence we impose that the number density of phonon to be of $O(1)$, whereas $g$ shares the order with $\hbar \omega D(\omega)$.


\bibliographystyle{siamplain}

\end{document}


\maketitle

\section{Proof of Proposition~2.2}\label{appen-proof2.2}
\begin{proof}
We will use the following ansatz of former expansion: $g_\eps :=g_0 + \eps g_1 + \eps^2 g_2 + ...$ into~(2.15) and equate the terms with like orders in $\varepsilon$, we have 
\begin{align*}
O(\eps^{-2}):\quad \mathcal{L}_0[g_0]  =\dfrac{\l g_0/\tau\r_{\mu,\omega}}{\l g^*/\tau\r_{\mu, \omega}} g^*(\omega) - g_0 = 0\,.
\end{align*}
This suggests that $g_0$ has no $\mu$-dependence, and behaves as equilibrium on $\omega$, i.e. 
\begin{equation}\label{eqn: g0-u}
g_0(t,x,\mu, \omega) = g^*(\omega)u(t,x)\,,
\end{equation}
where $u(t,x)$ is some function of $t, x$ to be specified later. 
\begin{align*}
O(\eps^{-1}): \quad \mu v(\omega) \p_x g_0 = \dfrac{1}{\tau(\omega)} \mathcal{L}_0[g_1]\,.
\end{align*}
Similar to property (5) of $\mathcal L$ in Proposition~2.2, we can find a unique inverse of $\mathcal{L}_0^{-1}$ in the space of $\left(\text{Ker}\mathcal{L}_0\right)^\perp$, making:
\begin{equation}\label{eqn: g1-u}
g_1(t,x,\mu,\omega) = -\mu v(\omega) \tau(\omega) \p_x g_0 =  -\mu v(\omega) \tau(\omega) g^*(\omega) \p_x u \,.
\end{equation}
At last, 
\begin{align*}
O(\eps^{0}): \quad \partial_t g_0 + \mu v(\omega) \p_x g_1 = \dfrac{1}{\tau(\omega)} \mathcal{L}_0 [g_2]\,.
\end{align*}
Using the conservation of $\mathcal L$ (see property (2) in Proposition~2.2), we have:
\begin{equation*}
\average{ \partial_t g_0 + \mu v(\omega) \p_x g_1 }_{\mu,\omega} = \average{ \dfrac{1}{\tau(\omega)}\mathcal{L}_0 [g_2]}_{\mu,\omega} = 0\,.
\end{equation*}
Substituting $g_0$ and $g_1$  expressed in $u(t,x)$ \eqref{eqn: g0-u} \eqref{eqn: g1-u} into the above equation, it becomes
\begin{equation*}
\begin{aligned}
\l \partial_t g_0 + \mu v(\omega) \p_x g_1 \r_{\mu,\omega} & = 
\average{ (\partial_t u) g^* - \mu^2 v^2 \tau g^*\p_x (\p_x u) }_{\mu,\omega}\,, \\
& = \partial_t u \l g^* \r_{\mu,\omega} - \int_\mu \mu^2 d\mu \int_0^\infty \tau v^2 g^* (\p_x^2 u) d\omega\,, \\
& = \l g^* \r_{\mu, \omega} \partial_t u - \frac{1}{3}\l \tau v^2 g^* \r_{\mu,\omega} (\p_x^2 u) = 0\,.
\end{aligned}
\end{equation*}
recovering~(2.16).

To show~(2.17), we note that 
\[
g_\varepsilon \approx g_0 + \varepsilon g_1 = g^*(\omega) u(t,x) - \eps \mu v(\omega) \tau(\omega) g^*(\omega) \p_x u\,,
\]
and thus by the definition of heat flux~(2.7) and temperature~(2.8), we have 
\begin{align*}
q_\eps(t,x) &:= \frac{1}{\eps} \average{\mu v(\omega)g_\eps}_{\mu,\omega} = - \frac{1}{3}\average{v^2 \tau g^*}_\omega \p_x u + O(\eps)\,,\\
T_\eps(t,x) &:= \frac{\average{g_\eps/\tau}_{\mu,\omega}}{\average{g^*/\tau}_{\mu,\omega}} = u(t,x)+ O(\eps)\,.
\end{align*}
Passing the limit $\eps\to 0$ formally, by definition of conductivity~(2.9), we recover:
\begin{align}\label{eqn: heat-con2}
\lim_{\eps\to0}\kappa_\eps := - \lim_{\eps\to0}\frac{q_\eps}{\p_x T_\eps} = \frac{1}{3}\average{v^2 \tau g^*}_\omega \,,
\end{align}
finishing the proof of~(2.7).
\end{proof}

\section{Convergence of SGD}\label{sec: SGD convergence}
In this subsection we expand our discussion about the convergence of SGD algorithms (Theorem~4.1). Recall that our loss function is in the form of:
\[F[\tau] = \sum_{i=1}^N f_i[\tau]\,,\]
where each $f_i$ is corresponding to different source-test function pair: $\{\phi_i, \psi_i\}.$
And at each step, only one gradient is computed: $\nabla f_{\xi_k}[\tau^k]$, with $\xi_k \sim \text{Unif}\{1,...,N\}$.
\begin{theorem}
(Adpated from~\cite{JZZ-SGD-2020}) Under the following two assumptions:
\begin{itemize}
\item Lipschitz objective gradients. There exists some $L>0$ such that for any $\tau, \hat{\tau}$ that are bounded from above and below:
\begin{align}
\|\nabla F[\tau] - \nabla F[\hat{\tau}]\|_2 \leq L\|\tau - \hat{\tau}\|_2\,.
\end{align}
\item First and second moment limits. There exist $\mu_G \geq \mu>0, M\geq 0$ and $M_V \geq 0$ such that for all $k\in \mathbb{N}$,
\begin{align}
\nabla F[\tau^k]^T \mathbb{E}_{\xi_k}[\nabla F[\tau^k]] \geq \mu \|\nabla F[\tau^k]\|_2^2\,, \label{eqn: SGD-cond-1}\\
\mathbb{E}_{\xi_k}[\nabla F[\tau^k]] \leq \mu_G \|\nabla F[\tau^k]\|_2\,,\label{eqn: SGD-cond-2}\\
\mathbb{V}_{\xi_k}[\nabla f_{\xi_k}[\tau^k]] \leq M + M_V \|\nabla F[\tau^k]\|_2^2\,.\label{eqn: SGD-cond-3}
\end{align}
\end{itemize}
Then with a properly tuned sequence of step-size, SGD algorithm converges, i.e.
\[ \liminf_k \,\mathbb{E}[\|\nabla F[\tau]\|_2^2] = 0 \,. \]
\end{theorem}

To verify assumption~\eqref{eqn: SGD-cond-1}, we only need to prove for any $i: \{\phi_i, \psi_i\}$, $\nabla f_i[\tau]$ is Lipschitz continuous. Below the sub-index $i$ is dropped for simplicity.
\begin{lemma} Under some boundedness assumptions, at fixed $\tau$ and for small perturbation $\tilde \tau$, the Fr\'echet derivative $\nabla f[\tau]$ (or equivalently, $\frac{\delta L}{\delta\tau}$ in~(4.12)) is Lipschitz, i.e. there exists some $C>0$ such that 
\begin{equation}\label{eqn: df-Lipschitz}\|\nabla f[\tau + \tilde \tau] - \nabla f[\tau]\|_2 \leq C\|\tilde \tau\|_2\,.\end{equation}
\end{lemma}
\begin{proof}
Suppose both $\tau$ and $g^*$ are both lower and upper bounded in the domain of $\omega$ of our concern: 
\[0<c_1<\tau<c_2\,, 0< c_3< g^*<c_4\,,\]
We will prove that the perturbation of gradient $\nabla f[\tau]$ induced by $\tilde \tau$ is bounded. Recall the gradient formula~(4.12):
\begin{align*}
\nabla f[\tau] = & \dfrac{l}{\average{h^*}_\omega^2} \dfrac{h^*}{\tau} \langle h \psi\rangle_{t,x=0, \mu,\omega} + \dfrac{1}{\tau} \average{p \dfrac{\langle h \rangle_{\mu,\omega}}{\langle h^* \rangle_{\omega}}}_{t, x, \mu} - \dfrac{h^*}{\tau} \dfrac{1}{\langle h^* \rangle_{\omega}^2} \langle \langle h \rangle_{\mu,\omega}\langle p \rangle_{\mu,\omega}\rangle_{t, x}\\
& + \frac{1}{g^*}\langle \mathcal{L}[h]p\rangle_{t,x,\mu}
+ \frac{1}{g^*}\langle \phi \mu v p\rangle_{t,x=0,\mu>0} - \frac{1}{\tau^2\langle h^*\rangle_\omega}\langle \phi \psi\rangle_{t,\mu>0}\,,
\end{align*}
from which we can write out the expression of $\nabla f[\tau+\tilde{\tau}] - \nabla f[\tau]$.
Upon linear approximation, the perturbation can be written into
\begin{align*}
\nabla f[\tau+\tilde{\tau}] - \nabla f[\tau] & = L_1[\tilde{\tau}] + L_2[\tilde{h}] + L_3[\tilde{p}]\,,
\end{align*}
where $L_1$ is a linear function of $\tilde{\tau}$ and involves the unperturbed $\{h, p, h^*, \tau, \phi, \psi\}$, similarly for $L_2$ and $L_3$. Here $\tilde{h}$ is the perturbation to the forward solution and solves~(4.13). We will denote the source term in~(4.13) as 
\begin{equation}\label{eqn: source-Sh}
S_h[\tilde \tau](t,x,\mu,\omega):= - \dfrac{\tilde \tau}{\tau}\mathcal{L}[h] +\dfrac{\average{h}_{\mu,\omega}}{\average{h^*}_{\omega}} \left(-\dfrac{\tilde \tau}{\tau} h^* + \dfrac{\average{h^* \tilde{\tau}/\tau}_\omega}{\average{h^*}_\omega}h^*\right)\,,\end{equation}
then $\tilde h$ solves:
\begin{equation}\label{eqn: perturb-2}
\partial_t \tilde{h} + \mu v \partial_x \tilde{h} = \dfrac{1}{\tau}\mathcal{L}[\tilde{h}] +  \dfrac{1}{\tau} S_h[\tilde \tau]\,.
\end{equation}
Likewise, $\tilde{p}$ is the perturbation to the adjoint solution~(4.11) and solves the following system:
\begin{align}\label{eqn: adjoint-perturb}\left\{\begin{array}{ll}
\partial_t \tilde{p} + \mu v \partial_x \tilde{p} = - \dfrac{1}{\tau} \mathcal{L}[\tilde{p}] - \dfrac{1}{\tau} S_p[\tilde \tau]\,, & 0<x<1\\[3mm]
\tilde{p}(t=T,x,\mu,\omega) = 0\,, & \mu<0\\[3mm]
\tilde{p}(t,x=0,\mu,\omega) = \dfrac{\psi}{\mu v}\left(\tilde l\dfrac{h^*}{\average{h^*}_\omega} - \dfrac{\tilde{\tau}}{\tau}h^* l + \dfrac{(h^*)^2 \tilde{\tau}}{\average{h^*}_\omega^2 \tau}l\right) \,, & \mu<0\\[3mm]
\tilde{p}(t,x=1,-\mu,\omega) = \tilde{p}(t,x=1,\mu,\omega)\,, & \mu<0
\end{array}\right.\end{align}

Using the maximal principle proved in~\cite{GambaLiNair} that
\[\|h\|_\infty \leq C\|\phi\|_\infty\,, \|p\|_\infty \leq C \|\psi\|_\infty\,,\]
it is easy to check that $L_1$ is Lipschitz continuous with respect to $\tilde \tau$:
\[\|L_1[\tilde{\tau}]\|_2 \leq C \|\tilde{\tau}\|_2\,.\]
And the constant $C = C(\|\phi\|_{\infty}, \|\psi\|_{\infty}, c_1, c_2, c_3, c_4)$.

Then we will prove Lipschitz continuity of $\tilde h$ with respect to $\tilde \tau$. As $L_2$ is linearly dependent in $\tilde h$, it is then Lipschitz contibuous in $\tilde \tau$. Firstly we note that $S_h$~\eqref{eqn: source-Sh} is Lipschitz with respect to $\tilde \tau$ by Cauchy-Schwatz inequality:
\[\|S_h[\tilde \tau]\|_2 \leq C\|\tilde \tau\|_2 \,,\]
where the constant $C = C(\|\phi\|_\infty, c_1, c_2, c_3, c_4).$ 

In addition, we multiply $L[\tilde h]$ by $\frac{\tilde{h}}{h^*}$ to attain the following inequality: 
\begin{align*}
\average{\mathcal{L}[\tilde h] \frac{\tilde{h}}{h^*}}_{x,\mu,\omega}
& = \frac{1}{\average{h^*}_\omega}\average{\average{\tilde h}_{\mu,\omega}^2}_{x} - \average{\frac{\tilde{h}^2}{h^*}}_{x,\mu,\omega}\,,\\
& \leq \frac{1}{\average{h^*}_\omega} \average{\frac{\tilde{h}^2}{h^*}}_{x,\mu,\omega} \average{h^*}_{x,\mu,\omega} - \average{\frac{\tilde{h}^2}{h^*}}_{x,\mu,\omega} \leq 0\,,
\end{align*}
the last line is because of Cauchy-Schwatz inequality and 
\[\langle 1\rangle_{x,\mu} = \int_0^1 \int_{-1}^1 1 \,dxd\mu = 1\,.\]

Then equation~\eqref{eqn: perturb-2} is multiplied by $\frac{\tau \tilde h}{h^*}$ and taken the average of $\average{\cdot}_{x,\mu,\omega}$ on both sides, by the above inequality we have:
\begin{align*}
\average{\frac{\tau}{h^*} \frac{1}{2} \partial_t \tilde{h}^2}_{x,\mu,\omega} + \average{\frac{\mu v \tau}{h^*} \frac{1}{2} \partial_x \tilde{h}^2}_{x,\mu,\omega} & = \average{\mathcal{L}[\tilde h] \frac{\tilde{h}}{h^*}}_{x,\mu,\omega} + \average{S_h[\tilde \tau] \frac{\tilde{h}}{h^*} }_{x,\mu,\omega}\,,\\
C\partial_t \average{\tilde{h}^2}_{x,\mu,\omega} + \average{\frac{\mu v \tau}{h^*} \partial_x \tilde{h}^2}_{x,\mu,\omega} & \leq 2\average{S_h[\tilde \tau] \frac{\tilde{h}}{h^*}}_{x,\mu,\omega} \,,
\end{align*}
here $C = C(c_1, c_4)>0.$ By integrating in $x$, the second term at LHS can be written out,
\begin{align*}
\average{\frac{\mu v \tau}{h^*} \partial_x \tilde{h}^2}_{x,\mu,\omega} & = \average{\frac{\mu v \tau}{h^*} \tilde{h}^2}_{x=1,\mu,\omega} - \average{\frac{\mu v}{h^*} \frac{\phi^2 }{\tau^3} \tilde{\tau}^2 }_{\mu^+,\omega} - \average{\frac{\mu v \tau}{h^*} \tilde{h}^2 }_{\mu^-,\omega}\,.
\end{align*}
At $x=1$ we incorporate the reflective boundary condition in~(4.13) to get,
\begin{align*}
\average{\frac{\mu v \tau}{h^*} \tilde{h}^2}_{x=1,\mu,\omega} = 0\,,
\end{align*}
and simply realizing that the third term is non-negative, we have
\[\average{\frac{\mu v \tau}{h^*} \partial_x \tilde{h}^2}_{x,\mu,\omega} \geq - \average{\frac{\mu v}{h^*} \frac{\phi^2 }{\tau^3} \tilde{\tau}^2 }_{\mu^+,\omega}\,.\]
Therefore,
\begin{align*}
C \frac{d}{dt} \average{\tilde{h}^2}_{x,\mu,\omega} & \leq \average{\frac{\mu v}{h^*} \frac{\phi^2}{\tau^3} \tilde{\tau}^2}_{\mu^+,\omega} + 2\average{S_h[\tilde \tau] \frac{\tilde{h}}{h^*}}_{x,\mu,\omega}\,,\\
\frac{d}{dt} \|\tilde{h}\|_2^2 & \leq C_1 \|\tilde \tau \|_2^2 + C_2 \|\tilde \tau\|_2 \|\tilde h \|_2\,, \\
\frac{d}{dt} \|\tilde{h}\|_2^2 & \leq C_1 \|\tilde \tau \|_2^2 + C_2 \|\tilde h \|^2_2\,,
\end{align*}
the second inequality is by Cauchy-Schwarz inequality and the third is by Young's inequality. We define 
\[\|f\|_2(t) := \left(\int f^2 \, dxd\mu d\omega\right)^{1/2}\,,\]
in particular, 
\[\|\tilde \tau\|_2 = \left(\int \tilde \tau^2 \,d\omega\right)^{1/2}\,.\]
We are now ready to apply Gronwall's inequality to the ODE above and obtain,
\begin{align*}
\|\tilde h \|_2^2(t) &\leq \|\tilde h \|_2^2(0)e^{\frac{C}{2}t} + C\|\tilde \tau\|_2^2 \left(e^{\frac{C}{2}t} -1 \right)\,,\\
& \leq C \|\tilde \tau\|_2^2\,, & C = C(T) \text{ for all } 0\leq t \leq T\,.
\end{align*}
hence 
\[\|L_2[\tilde h]\|_2 \leq C \|\tilde \tau\|_2\,.\]

Similar argument can be applied to $\tilde p$ and attain: 
\[
\|\tilde p\|_2 \leq C \|\tilde \tau\|_2\,.
\]

Hence we have proved $L_1, L_2, L_3$ are all Lipschitz with respect to $\tilde \tau$, so there exists some constant $C$ depending on the source function $\phi$, test function $\psi$, domain of $\tau$ and final time $T$, such that~\eqref{eqn: df-Lipschitz} holds.

\end{proof}

With regard to first and second moment limits, it is obvious to check that our method satisfies~\eqref{eqn: SGD-cond-1}, \eqref{eqn: SGD-cond-2} with $\mu = \mu_G = 1$. For~\eqref{eqn: SGD-cond-3}, it can be further transferred to two conditions on the angles between $\nabla f_i$ and their amplitudes. 
\begin{lemma} Under the following two conditions, the second moment limit assumption~\eqref{eqn: SGD-cond-3} can be fulfilled:
\begin{itemize}
\item Angle is small: there exists some constant $c_1$ (uniform in $i$), such that for any $i,j$:
\begin{align}\label{eqn: angle cond}
\frac{\langle \nabla f_i, \nabla f_j\rangle_\omega}{\|\nabla f_i\|_2 \,\|\nabla f_j\|_2} \geq c_1\,,     
\end{align}
and the inner product is in $L^2$:
\begin{align*}
\langle \nabla f_i, \nabla f_j\rangle_\omega := \int \nabla f_i \, \nabla f_j\, d\omega\,.
\end{align*}
\item Amplitudes among all gradients are comparable. There exist some constants $\alpha, \beta >0$ (uniform in $i, j$), such that:
\begin{align}\label{eqn: amp cond}
\alpha \|\nabla f_j\|_2 \leq \|\nabla f_i\|_2 \leq \beta \|\nabla f_j\|_2\,.
\end{align}
\end{itemize}
\end{lemma}
\begin{proof}
We first rewrite condition~\eqref{eqn: SGD-cond-3} by definition.
\begin{align*}
\mathbb{V}_{\xi_k}[\nabla f_{\xi_k}] = \mathbb{E}_{\xi_k}[\|\nabla f_{\xi_k}\|_2^2] - \|\mathbb{E}_{\xi_k}(\nabla f_{\xi_k})\|_2^2 = \frac{1}{N}\sum_{i=1}^N \|\nabla f_{i}\|_2^2 - \|\nabla F\|_2^2\,.
\end{align*}
Thus condition~\eqref{eqn: SGD-cond-3} is equivalent to: $\exists M, M_V\geq 0$, such that
\[\frac{1}{N}\sum_{i=1}^N \|\nabla f_i\|_2^2  \leq M + \frac{M_v+1}{N^2}(\sum_{i=1}^N\, \|\nabla f_i\|_2^2 + \sum_{i\not=j} \langle\nabla f_i, \nabla f_j\rangle_\omega)\,.\]
Assuming \eqref{eqn: angle cond} and \eqref{eqn: amp cond},
\begin{align*}
\sum_{i\not=j} \langle \nabla f_i, \nabla f_j\rangle_\omega &\geq c_1 \sum_{i\not=j} \|\nabla f_i\|_2 \|\nabla f_j\|_2\,,\\
& \geq c_1 \sum_{i\not=j} \|\nabla f_i\|_2 \frac{1}{\beta} \|\nabla f_i\|_2 = \frac{c_1}{\beta} (N-1)\sum_i \|\nabla f_i\|_2^2\,.
\end{align*}
Then we consider and denote $c_2 = \min\{1, \frac{c_1}{\beta}\}$,
\begin{align*}
\frac{1}{N^2}(\sum_{i=1}^N \|\nabla l_i\|_2^2 + \sum_{i\not=j} \langle \nabla f_i, \nabla f_j\rangle_\omega) & \geq \frac{1}{N^2} (\sum_{i=1}^N \|\nabla f_i\|_2^2 + \frac{c_1}{\beta} (N-1)\sum_i \|\nabla f_i\|_2^2)\,,\\
& \geq \frac{c_2}{N^2}N(\sum_i \|\nabla f_i\|_2^2)\,,\\
& \, = \frac{c_2}{N} \sum_i \|\nabla f_i\|_2^2\,.
\end{align*}
Let $M=0$ and $M_V + 1 = \frac{1}{c_2}$ then the second moment limit is proved.
\end{proof}
Due to the complicated expression of gradient in our problem, the geometric meaning of angle between two gradient is rather unclear, we leave the detailed discussion of these two conditions for future study but show some numerical demonstration. Intuitively speaking, suppose among all vectors of $\nabla f_i$, the angle between any two vectors are small, which means they are pointing almost to the same direction, the variance should be controlled.

Despite our observation of gradients (Figure~8) that they concentrate in different $\omega$, since non-trivial linear combination will span the same space as the original basis, we can consider some linear combination of the set $\{\nabla f_i\}$ such that they satisfy these two conditions.

We will show an example of finding a linear combination of gradients at a fix $\tau$. For simplicity, we take $\tau$ to be at the initial guess, 
\[\tau^0(\omega) = -0.15 (\omega - 4) + 1.4\,.\]
and at this specific $\tau^0(\omega)$, we can gather the gradients with respect to each $f_i$ (note that it is corresponding to different pair of $\{\phi_i, \psi_i\}$) into a matrix:
\begin{align*}
\nabla F[\tau^0] = \left[\begin{array}{ll}
\nabla f_1[\tau^0]\,, ...\,, \nabla f_{10}[\tau^0] \end{array}
\right]\,.\end{align*}
Our goal is to find a linear combination of $\{\nabla f_i[\tau]\}$ such that the angle between any two vectors is not $0$ (or not too small).


Numerical norm is computed for each individual gradient where the ratio of norms is between $\frac{1}{2}$ and $2$, the norm of this finite set of gradient vectors is obviously comparable.
\begin{figure}[htbp]
    \centering
    \includegraphics[scale =0.16]{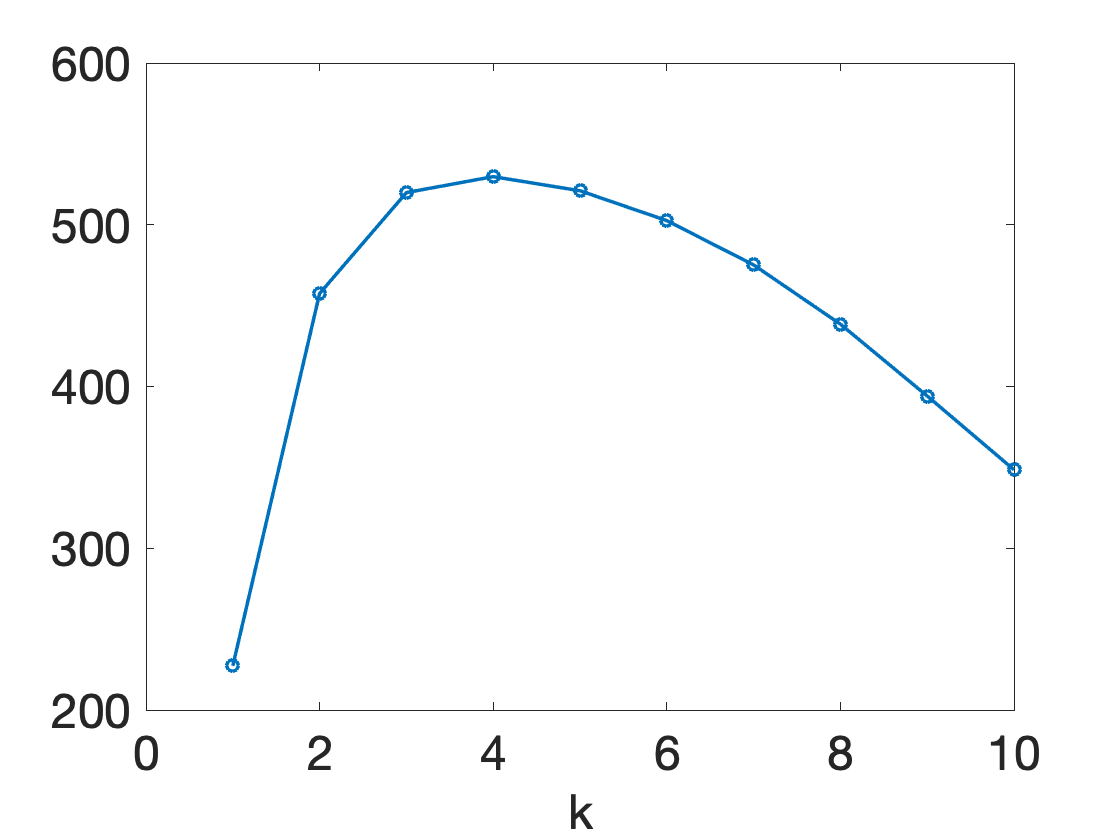}
    \caption{$\|\nabla L_i\|_2\,$ corresponding to Figure~8.}
    \label{fig: show gradient norms}
\end{figure}

And we also compute the angle matrix, whose entry represents the cosine value between any two gradient vectors from $\{\nabla f_i\}$. It is noticeable that the angle between $\nabla f_1, \nabla f_3$ is almost orthogonal that
\[\frac{\average{\nabla f_1, \nabla f_3}}{\|\nabla f_1\|\|\nabla f_3\|} = -0.01\,,\]
and other cosine values are mostly between $0.1$ and less than $0.7$.

Let $A$ be a uniform random matrix, i.e. 
\[A = (a_{ij})\,, a_{ij}\sim \text{Unif}(0,1)\,,\]
and right multiply it to gradient matrix: $\tilde{G} = GA\,,$ we obtain a new set of vectors, as shown in Figure~\ref{fig: show new gradients}. Intuitively, this kind of recombination averages out the update on each specific point $\omega$ and would give a smaller range of angles between each pair of $\nabla f_i$.

\begin{figure}[htbp]
    \centering
    \begin{subfigure}[b]{0.45\textwidth}
         \centering
         \includegraphics[width=\textwidth]{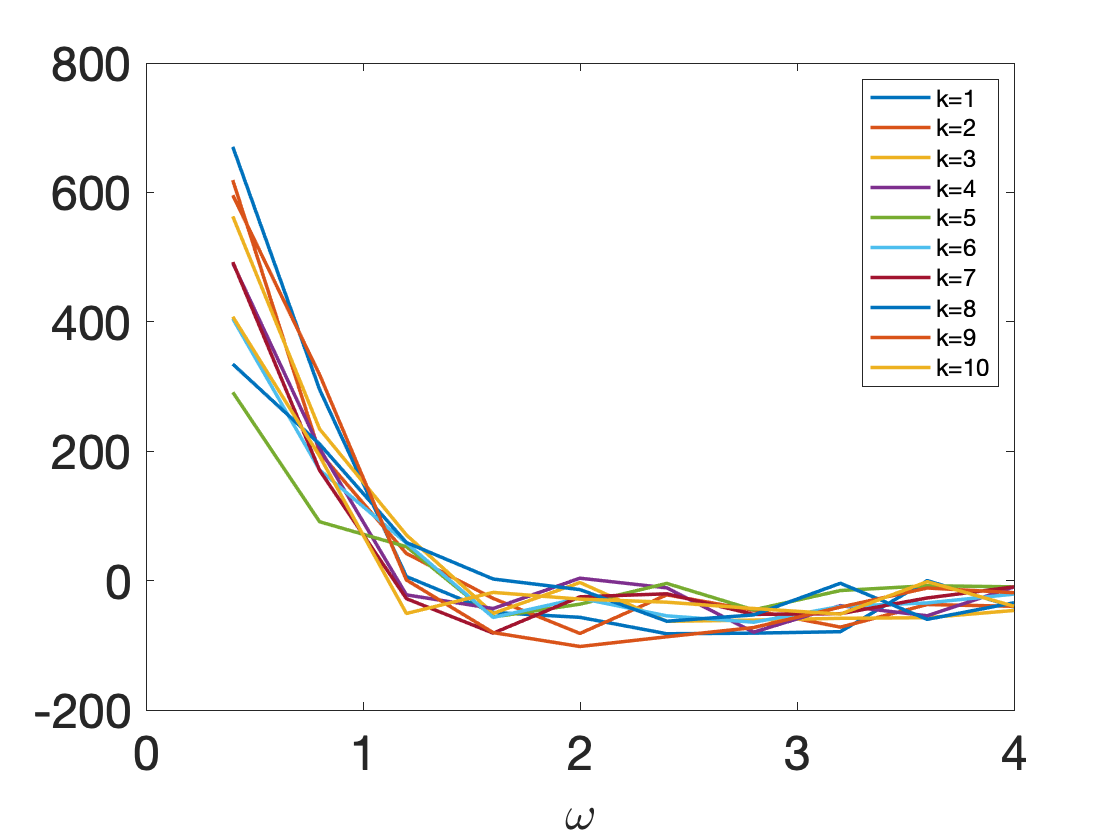}
         \caption{uniform linear combination}
    \end{subfigure}
    \hfill
    \begin{subfigure}[b]{0.45\textwidth}
         \centering
         \includegraphics[width=\textwidth]{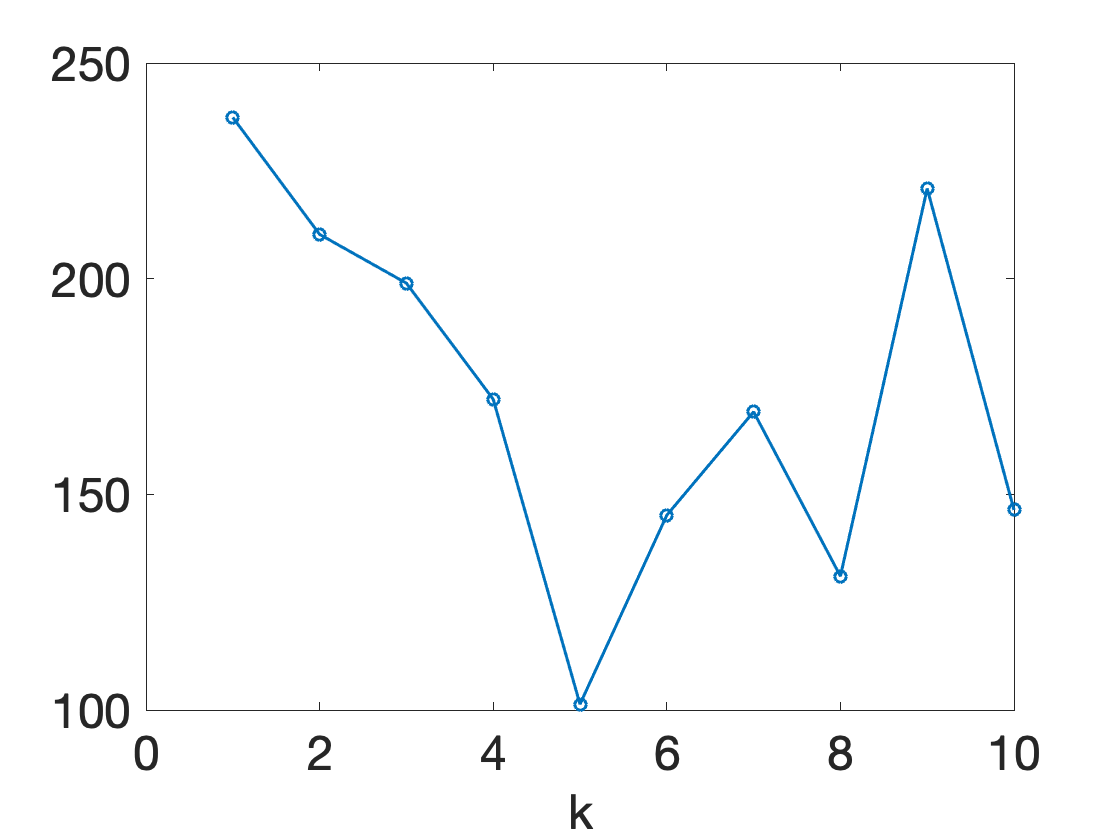}
         \caption{norm}
    \end{subfigure}
    \caption{Recombination.}
    \label{fig: show new gradients}
\end{figure}

We note that the ratio of norms between these $\{\nabla \tilde{f}_i\}$ roughly stays between $0.5$ and $1$, while the angle between any two vectors are now small in the sense that their cosine values are mostly between $0.92$ and $1$.

\bibliographystyle{siamplain}